\documentclass[12pt,reqno]{amsart}
\usepackage{amsmath, amsthm, amssymb}
\usepackage[margin=2cm]{geometry}
\usepackage{epsfig}
\usepackage{rawfonts}
\usepackage{enumerate}
\usepackage{graphics}
\usepackage{multirow}
\usepackage{xspace}
\usepackage{graphicx}
\usepackage{pgf,tikz,pgfplots}
\usepackage{mathrsfs}
\usetikzlibrary{arrows}
\usepackage{amsmath}
\usepackage{amsfonts}
\usepackage{amssymb}
\usepackage{amsthm}
\usepackage{graphicx}
\usepackage{booktabs}
\usepackage{caption}
\usepackage{listings}
\usepackage{setspace}
\usepackage[mathscr]{eucal}
\usepackage{pgfplots}
\usepackage{hyperref}
\usepackage{wrapfig}
\usepackage{floatflt,epsfig}
\usepackage{ dsfont }
\usepackage{amscd}
\usepackage{tikz-cd}
\usepackage{fancyhdr}
\usepackage[all]{xy}
\usepackage{latexsym}
\usepackage{amscd}
\usepackage{pifont}
\usepackage{eufrak}
\usepackage{subfig}
\usepackage{easyReview}

\newcommand{\rank}{\mathop{\rm rank}\nolimits}

\newcommand{\numberset}{\mathbb}

\newcommand{\Z}{\numberset{Z}}

\def\NN{{\mathbb N}}
\def\ZZ{{\mathbb Z}}

\newcommand{\lt}{\mathop{\rm in}\nolimits}

\newcommand{\cB}{\mathcal{B}}
\newcommand{\cC}{\mathcal{C}}

\newcommand{\cP}{\mathcal{P}}
\newcommand{\cH}{\mathcal{H}}
\newcommand{\cI}{\mathcal{I}}
\newcommand{\cK}{\mathcal{K}}

\newcommand{\cF}{\mathcal{F}}
\newcommand{\cQ}{\mathcal{Q}}
\newcommand{\cW}{\mathcal{W}}
\newcommand{\cR}{\mathcal{R}}
\newcommand{\cS}{\mathcal{S}}
\newcommand{\cT}{\mathcal{T}}
\newcommand{\cL}{\mathcal{L}}

\newcommand{\cG}{\mathcal{G}}
\newcommand{\rHP}{\mathrm{HP}}

\theoremstyle{plain}

\theoremstyle{theorem}

\newtheorem{defn}{Definition}[section]
\newtheorem{prop}[defn]{Proposition}
\newtheorem{thm}[defn]{Theorem}
\newtheorem{lemma}[defn]{Lemma}
\newtheorem{conj}[defn]{Conjecture}
\newtheorem{coro}[defn]{Corollary}
\newtheorem{exa}[defn]{Example}
\newtheorem{rmk}[defn]{Remark}

\newtheorem{disc}[defn]{Discussion}

\theoremstyle{remark}

\begin{document}

	\title[Shellable simplicial complex and rook polynomial of frame polyominoes]{Shellable simplicial complex and switching rook polynomial of frame polyominoes}
	    \author{RIZWAN JAHANGIR}
			\address{Sabanci University, Faculty of Engineering and Natural Sciences, Orta Mahalle, Tuzla 34956, Istanbul, Turkey}
			\email{rizwan@sabanciuniv.edu}
   
			\author{FRANCESCO NAVARRA}
			\address{Università di Messina, Dipartimento di Scienze Matematiche e Informatiche, Scienze Fisiche e Scienze della Terra\\
			Viale Ferdinando Stagno D'Alcontres 31\\
			98166 Messina, Italy}
			\email{francesco.navarra@unime.it}

	

	\begin{abstract}
    Let $\cP$ be a frame polyomino, a new kind of non-simple polyomino. In this paper we study the $h$-polynomial of $K[\cP]$ in terms of the switching rook polynomial of $\cP$ using the shellable simplicial complex $\Delta(\cP)$ attached to $\cP$. We provide a suitable shelling order for $\Delta(\cP)$ and we define a bijection between the set of the canonical configuration of $j$ rooks in $\cP$ and the facets of $\Delta(\cP)$ with $j$ steps. Finally we use a well-known combinatorial result, due to McMullen and Walkup, about the $h$-vector of a shellable simplicial complex to interpret the $h$-polynomial of $K[\cP]$ as the switching rook polynomial of $\cP$. 		
 	\end{abstract}

	\maketitle
	
	\section{Introduction}
	
     \noindent A central topic in Combinatorial Commutative Algebra is the study of the ideals of $t$-minors of an $m \times n$ matrix of indeterminates, for any integer 
	$1\leq t\leq \min\{m,n\}$. The most known references are \cite{conca1}, \cite{conca2}, \cite{conca3}, \cite{gorla} for the ideals
	generated by all $t$-minors of a one or two-sided ladder, \cite{adiajent1}, \cite{adjent 2}, \cite{adiajent3} for the ideals of adjacent 2-minors and \cite{2.n} for the ideals generated by an arbitrary set of $2$-minors in a $2\times n$ matrix.  In \cite{Qureshi} A.A. Qureshi defines a new class of binomial ideals that arises from $2$-minors of a polyomino, generalizing the class of the ideals generated by 2-minors of $m\times n$ matrices.\\
    A polyomino is a finite collection of unitary squares joined edge by edge. This term was coined by Golomb in 1953 for the first time and several combinatorial aspects of polyominoes are discussed in his monograph \cite{golomb}. If $\cP$ is a polyomino, then the \textit{polyomino ideal} of $\cP$ is the ideal generated by all inner 2-minors of $\cP$ in a suitable polynomial ring $S_\cP$ and it is denoted by $I_{\cP}$. The study of the main algebraic properties of the quotient ring $K[\cP]=S_\cP/I_{\cP}$, depending on the shape of $\cP$, is the purpose of this research. The primality of the polyomino ideal is studied in many articles (see \cite{Cisto_Navarra_closed_path}, \cite{Cisto_Navarra_weakly}, \cite{def balanced}, \cite{Not simple with localization}, \cite{Trento}, \cite{Trento2}, \cite{Simple are prime}) but a complete characterization for the prime polyomino ideals is still unknown, as well as for the Cohen-Macaulayness and the normality of $K[\cP]$ (\cite{Cisto_Navarra_CM_closed_path}) or for the K\"onig type property of $I_{\cP}$ (see \cite{Dinu_Navarra_Konig}, \cite{Hibi - Herzog Konig type polyomino}). Consult \cite{Andrei}, \cite{Cisto_Navarra_Veer}, \cite{Herzog rioluzioni lineari}, \cite{Kummini CD} for other interesting algebraic properties of the polyomino ideal. \\
    Recently a new line of research has been developed, regarding the study of the Hilbert-Poincar\'e series and the Castelnuovo-Mumford regularity of $K[\cP]$ in terms of the rook polynomial of $\cP$ and its invariants. The rook polynomial of a polyomino $\cP$ is well-known in literature (see \cite{Kapl-Riordan} and \cite[Chapter 7]{Riordan}). It is a polynomial $r_{\cP}(t)=\sum_{i=0}^n r_i t^i\in \ZZ[t]$, whose coefficient $r_i$ represents the number of distinct possibilities of arranging $i$ rooks on cells of $\cP$ in non-attacking positions (with the convention $r_0=1$); in particular, the degree of $r_{\cP}(t)$ is the rook number of $\cP$ which is the maximum number of rooks which can be arranged in $\cP$ in non-attacking positions. In \cite{L-convessi} the authors show that if $\cP$ is an $L$-convex polyomino then Castelnuovo-Mumford regularity of $K[\cP]$ is equal to the rook number of $\cP$. In \cite{Trento3} it is proved that for a simple thin polyomino the $h$-polynomial $h(t)$ of $K[\cP]$ is the rook polynomial of $\cP$ and the Gorenstein simple thin polyominoes are also characterized using the so-called $S$-property. Roughly speaking a polyomino is simple if it has no hole and it is thin if it does not contain the square tetromino. Finally it is conjectured that $r_{\cP}(t)=h(t)$ if and only if $\cP$ is thin (\cite[Conjecture 4.5]{Trento3}). Similar results are obtained in \cite{Cisto_Navarra_Hilbert_series} for a particular class of non-simple thin polyominoes, called closed paths, providing also partial support for that conjecture. In \cite{Kummini rook polynomial} the authors prove that for a convex polyomino $\cP$ whose vertex set is a sub-lattice of $\NN^2$ we have $h(t)\neq r_{\cP}(t)$ if $\cP$ is not thin. In \cite{Parallelogram Hilbert series} the authors devote their study to the Hilbert-Poincar\'e series of simple polyominoes which could be also non-thin. They generalize the rook polynomial of $\cP$, introducing a particular equivalence relation $\sim$ on the rook complex and a new polynomial $\tilde{r}_{\cP}(t)=\sum_{i=0}^n \tilde{r}_it^i$, which we call in this work \textit{switching rook polynomial} of $\cP$, where $\tilde{r}_i$ is the number of the equivalence classes with respect to $\sim$ of $i$ non-attacking rooks of $\cP$. In particular, they prove that $h(t)=\tilde{r}_{\cP}(t)$ for a parallelogram polyomino $\cP$ and, by a computational method, also for all simple polyominoes having rank at most eleven. Finally, they conjecture that the latter holds for all simple polyominoes (\cite[Conjecture 3.2]{Parallelogram Hilbert series}) and they provide a characterization of the parallelogram polyominoes whose coordinate ring is Gorenstein.\\
    Inspired by the above-mentioned results and conjectures, we study the Hilbert-Poincar\'e series of the coordinate ring of some non-simple polyominoes, providing a demonstrative technique whose key ingredient is the well-known result of McMullen and Walkup about the $h$-vector of a shellable simplicial complex (\cite[Corollary 5.1.14]{Bruns_Herzog}). This paper is arranged as follows. In Section 2 we recall some notions on polyominoes and we introduce the polyomino ideal following \cite{Qureshi}. In Section 3 we define the class of \textit{frame polyominoes}. Roughly speaking, a frame polyomino $\cP$ is a non-simple polyomino obtained by removing a parallelogram polyomino from a rectangle polyomino. Some algebraic properties of the $K$-algebras associated with a larger class of this kind of polyominoes are studied in \cite{Not simple with localization} and \cite{Shikama}, where the authors prove that the associated coordinate ring is a normal Cohen-Macaulay domain but without computing the Krull dimension. Let $\cP$ be a frame polyomino. In this section, we prove firstly that the Krull dimension of $K[\cP]$ is equal to the difference between the number of the vertices of $\cP$ and the number of the cells of $\cP$. After that, we start a deep study of the simplicial complex $\Delta(\cP)$ associated with the initial ideal of $I_{\cP}$ with respect to a suitable monomial order on $S_{\cP}$. We introduce a new combinatorial notion for a face of a simplicial complex attached to a polyomino, called a \textit{step}, and in Discussion \ref{Discussion: facet} we show in detail how a facet with $j$ steps of $\Delta(\cP)$ is nicely displayed in a frame polyomino $\cP$. Finally we show that the descending lexicographical order of the facets of $\Delta(\cP)$  forms a shelling order for $\Delta(\cP)$ and we reinterpret the theorem of McMullen-Walkup showing that the $i$-th coefficient of the $h$-polynomial of $K[\cP]$ is equal to the number of the facets of $\Delta(\cP)$ having $i$ steps (see Theorem \ref{Proposition: The facet less the lower right corner generates the intersection}).\\  
    In Section 4 we present the main result of this paper, which states that the $h$-polynomial of $K[\cP]$ is the switching rook polynomial of $\cP$ (see Theorem  \ref{mainthm}). At first, we recall some notions about the switching rook polynomial of a polyomino. After that, we show that every configuration of $j$ rooks in non-attacking position in $\cP$ can lead to a canonical configuration in $\cP$ and, in particular, we find a one-to-one correspondence between the facets with $j$ steps of $\Delta(\cP)$ and the canonical configurations in $\cP$ of $j$ rooks (see Theorem \ref{Theorem: bijection between facets and canonical positions}). The latter joined with the Theorem \ref{Proposition: The facet less the lower right corner generates the intersection}, allows us to prove our main goal. 
    We conclude the work showing that the Castelnuovo-Mumford regularity of $K[\cP]$ is the rook number of $\cP$ in Corollary \ref{Coro: rook number} and conjecturing that the main result of this paper holds for every polyomino (see Conjecture \ref{Conj}).

	\section{Polyominoes and polyomino ideals}\label{Section: Introduction}
	
	\noindent Let $(i,j),(k,l)\in \Z^2$. We say that $(i,j)\leq(k,l)$ if $i\leq k$ and $j\leq l$. Consider $a=(i,j)$ and $b=(k,l)$ in $\Z^2$ with $a\leq b$. The set $[a,b]=\{(m,n)\in \Z^2: i\leq m\leq k,\ j\leq n\leq l \}$ is called an \textit{interval} of $\Z^2$. 
	In addition, if $i< k$ and $j<l$ then $[a,b]$ is a \textit{proper} interval. In such a case we say $a, b$ the \textit{diagonal corners} of $[a,b]$ and $c=(i,l)$, $d=(k,j)$ the \textit{anti-diagonal corners} of $[a,b]$. If $j=l$ (or $i=k$) then $a$ and $b$ are in a \textit{horizontal} (or \textit{vertical}) \textit{position}. We denote by $]a,b[$ the set $\{(m,n)\in \Z^2: i< m< k,\ j< n< l\}$. A proper interval $C=[a,b]$ with $b=a+(1,1)$ is called a \textit{cell} of $\ZZ^2$; moreover, the elements $a$, $b$, $c$ and $d$ are called respectively the \textit{lower left}, \textit{upper right}, \textit{upper left} and \textit{lower right} \textit{corners} of $C$. The sets $\{a,c\}$, $\{c,b\}$, $\{b,d\}$ and $\{a,d\}$ are the \textit{edges} of $C$. We put $V(C)=\{a,b,c,d\}$ and $E(C)=\{\{a,c\},\{c,b\},\{b,d\},\{a,d\}\}$. Let $\cS$ be a non-empty collection of cells in $\Z^2$. The set of the vertices and of the edges of $\cS$ are respectively $V(\cS)=\bigcup_{C\in \cS}V(C)$ and $E(\cS)=\bigcup_{C\in \cS}E(C)$, while $\rank\cS$ is the number of cells belonging to $\cS$. If $C$ and $D$ are two distinct cells of $\cS$, then a \textit{walk} from $C$ to $D$ in $\cS$ is a sequence $\cC:C=C_1,\dots,C_m=D$ of cells of $\ZZ^2$ such that $C_i \cap C_{i+1}$ is an edge of $C_i$ and $C_{i+1}$ for $i=1,\dots,m-1$. In addition, if $C_i \neq C_j$ for all $i\neq j$, then $\cC$ is called a \textit{path} from $C$ to $D$. Two cells $C$ and $D$ are connected in $\cS$ if there exists a path of cells in $\cS$ from $C$ to $D$.
	A \textit{polyomino} $\cP$ is a non-empty, finite collection of cells in $\Z^2$ where any two cells of $\cP$ are connected in $\cP$. For instance, see Figure \ref{Figura: Polimino introduzione}.
	
	\begin{figure}[h]
		\centering
		\includegraphics[scale=0.6]{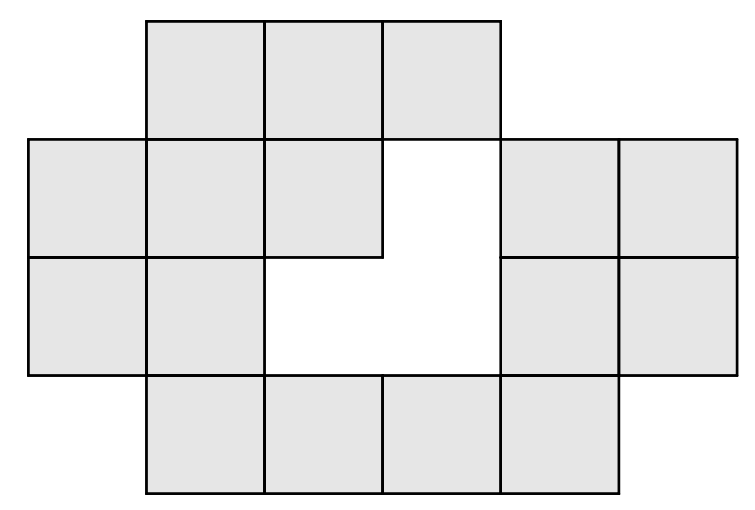}
		\caption{A polyomino.}
		\label{Figura: Polimino introduzione}
	\end{figure}
	
	\noindent A polyomino $\cP$ is \textit{simple} if for any two cells $C$ and $D$ not in $\cP$ there exists a path of cells not in $\cP$ from $C$ to $D$. A finite collection of cells $\cH$ not in $\cP$ is a \textit{hole} of $\cP$ if any two cells of $\cH$ are connected in $\cH$ and $\cH$ is maximal with respect to set inclusion. For example, the polyomino in Figure \ref{Figura: Polimino introduzione} is not simple with a hole. Obviously, each hole of $\cP$ is a simple polyomino and $\cP$ is simple if and only if it has no hole.\\
	Consider two cells $A$ and $B$ of $\Z^2$ with $a=(i,j)$ and $b=(k,l)$ as the lower left corners of $A$ and $B$ with $a\leq b$. A \textit{cell interval} $[A,B]$ is the set of the cells of $\Z^2$ with lower left corner $(r,s)$ such that $i\leqslant r\leqslant k$ and $j\leqslant s\leqslant l$. If $(i,j)$ and $(k,l)$ are in horizontal (or vertical) position, we say that the cells $A$ and $B$ are in a \textit{horizontal} (or \textit{vertical}) \textit{position}.\\
	Let $\cP$ be a polyomino. Consider two cells $A$ and $B$ of $\cP$ in a vertical or horizontal position. The cell interval $[A,B]$, containing $n>1$ cells, is called a
	\textit{block of $\cP$ of rank n} if all cells of $[A,B]$ belong to $\cP$. The cells $A$ and $B$ are called \textit{extremal} cells of $[A,B]$. Moreover, a block $\cB$ of $\cP$ is \textit{maximal} if there does not exist any block of $\cP$ which contains properly $\cB$. It is clear that an interval of $\ZZ^2$ identifies a cell interval of $\ZZ^2$ and vice versa, so we associate to an interval $I$ of $\ZZ^2$ the corresponding cell interval, which we denote by $\cP_{I}$. Note that $\cP_{I}$ is a polyomino itself. If $I=[a,b]$ and $c$, $d$ are the anti-diagonal corner of $I$, then the two cells containing $a$ or $b$ ($c$ or $d$) are called \textit{diagonal} (\textit{anti-diagonal}) cells of $\cP_I$. A proper interval $[a,b]$ is called an \textit{inner interval} of $\cP$ if all cells of $\cP_{[a,b]}$ belong to $\cP$. We denote by $\cI(\cP)$ the set of all inner intervals of $\cP$. An interval $[a,b]$ with $a=(i,j)$, $b=(k,j)$ and $i<k$ is called a \textit{horizontal edge interval} of $\cP$ if the sets $\{(\ell,j),(\ell+1,j)\}$ are edges of cells of $\cP$ for all $\ell=i,\dots,k-1$. In addition, if $\{(i-1,j),(i,j)\}$ and $\{(k,j),(k+1,j)\}$ do not belong to $E(\cP)$, then $[a,b]$ is called a \textit{maximal} horizontal edge interval of $\cP$. We define similarly a \textit{vertical edge interval} and a \textit{maximal} vertical edge interval. \\
	\noindent Let $\cP$ be a polyomino and  $S_\cP=K[x_v| v\in V(\cP)]$ be the polynomial ring of $\cP$ where $K$ is a field. If $[a,b]$ is an inner interval of $\cP$, with $a$,$b$ and $c$,$d$ respectively diagonal and anti-diagonal corners, then the binomial $x_ax_b-x_cx_d$ is called an \textit{inner 2-minor} of $\cP$. $I_{\cP}$ is called \textit{polyomino ideal} of $\cP$ and is defined as the ideal in $S_\cP$ generated by all the inner 2-minors of $\cP$. We denote by $G(\cP)$ the set of all generators of $I_{\cP}$. We set $K[\cP] = S_\cP/I_{\cP}$, which is the \textit{coordinate ring} of $\cP$. Along the paper, we work with the reverse lexicographical order $<$ on $S$ induced by the ordering of the variables defined by $x_{ij}>x_{kl}$ if $j > l$, or, $j = l$ and $i > k$. If $J$ is an ideal of $S_{\cP}$ we denote by $\lt_{<}(J)$ the initial ideal of $J$ with respect to $<$.\\ Now, we point out the conditions in order to the set of the generators of $I_{\cP}$ forms the reduced Gr\"obner basis of $I_{\cP}$ with respect to $<$.  The latter can be proved as done in \cite[Remark 4.2]{Qureshi} because, if $[p,q]$ is an inner interval of $\cP$ with anti-diagonal corners $r$ and $s$, then the initial monomial of $x_px_q-x_rx_s$ with respect to $<$ is given by $x_rx_s$.
	
	\begin{thm}
 \label{Theorem: Qureshi condition to have Groebner basis}
   	Let $\cP$ be a collection of cells. Then the
   	set of inner 2-minors of $\cP$ forms the reduced (quadratic) Gr\"obner basis with respect to $<$ if and only if for any two inner intervals $[b,a]$ and $[d,c]$ of $\cP$ with anti-diagonal corners $e, f$ and $f,g$ as shown
   	in Figure \ref{Figura: conditions for Groebner basis}, either $b, g$ or $e, c$ are anti-diagonal corners of an inner interval of $\cP$.
	\end{thm}

	\begin{figure}[h]
	\centering
	\includegraphics[scale=1]{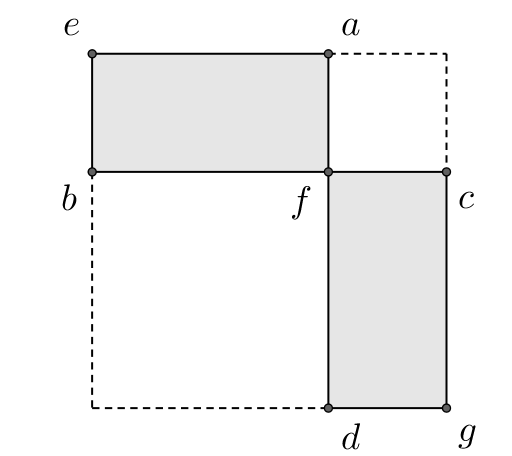}
	\caption{Conditions for Gr\"obner basis with respect to $<$.}
	\label{Figura: conditions for Groebner basis}
	\end{figure}

	\noindent We recall some notions on the Hilbert function and the Hilbert-Poincar\'{e} series of a graded $K$-algebra $R/I$. Let $R$ be a graded $K$-algebra and $I$ be an homogeneous ideal of $R$, then $R/I$ has a natural structure of graded $K$-algebra as $\bigoplus_{k\in\mathbb{N}}(R/I)_k$. The function $\mathrm{H}_{R/I}:\mathbb{N}\to \mathbb{N}$ with $\mathrm{H}_{R/I}(k)=\dim_{K} (R/I)_k$ is called the \textit{Hilbert function} of $R/I$ and the formal power series $\rHP_{R/I}(t)=\sum_{k\in\mathbb{N}}\mathrm{H}_{R/I}(k)t^k$ is defined as the \textit{Hilbert-Poincar\'e series} of $R/I$.
	It is known by Hilbert-Serre theorem that there exists a unique polynomial $h(t)\in \mathbb{Z}[t]$, called \textit{h-polynomial} of $R/I$, such that $h(1)\neq0$ and $\rHP_{R/I}(t)=\frac{h(t)}{(1-t)^d}$, where $d$ is the Krull dimension of $R/I$. Moreover, if $R/I$ is Cohen-Macaulay then $\mathrm{reg}(R/I)=\deg h(t)$. \\
	
	\noindent We recall some basic facts on simplicial complexes. A \textit{finite simplicial complex} $\Delta$ on $[n]:=\{1,\dots,n\}$ is a
	collection of subsets of $[n]$ satisfying the following two conditions:
	\begin{enumerate}
		\item if $F'\in \Delta$ and $F \subseteq F'$ then $F \in \Delta$;
		\item $\{i\}\in \Delta$ for all $i=1,\dots,n$.
	\end{enumerate}  The elements of $\Delta$ are called \textit{faces}, and the dimension of a face is one less than
	its cardinality. An \textit{edge} of $\Delta$ is a face of dimension 1, while a \textit{vertex} of $\Delta$ is a face
	of dimension 0. The maximal faces of $\Delta$ with respect to the set inclusion are called \textit{facets}. The dimension
	of $\Delta$ is given by $\max\{\dim(F):F\in \Delta\}$. A simplicial complex $\Delta$ is \textit{pure} if all facets have the same dimension. For a pure simplicial complex $\Delta$, the dimension of $\Delta$ is given trivially by the dimension of a facet of $\Delta$. Given a collection $\mathscr{F}=\{F_1,\dots,F_m\}$ of subsets of $[n]$, we denote by $\langle F_1,\dots,F_m\rangle$ or briefly $\langle \mathscr{F}\rangle$ the simplicial complex consisting of all subsets of $[n]$ which are contained in $F_i$ for some $i\in[m]$. This simplicial complex is said to be generated by $F_1,\dots,F_m$; in particular, if $\cF$ is the set of the facets of a simplicial complex $\Delta$, then $\Delta$ is generated by $\cF$. A pure simplicial complex $\Delta$ is called \textit{shellable} if the facets of $\Delta$ can be ordered as $F_1,\dots,F_m$ in such a way that $\langle F_1,\dots,F_{i-1}\rangle \cap \langle F_i\rangle$ is generated by a non-empty set of maximal proper faces of $\langle F_i\rangle$, for all $i\in \{2,\dots,m\}$ (see \cite[Definition 5.1.11]{Bruns_Herzog}). In such a case $F_1,\dots,F_m$ is called a \textit{shelling} of $\Delta$. \\
	Let $\Delta$ be a simplicial complex on $[n]$ and  $R=K[x_1,\dots,x_n]$ be the polynomial ring in $n$ variables over a field $K$. To every collection $F=\{i_1,\dots,i_r\}$ of $r$ distinct vertices of $\Delta$, there is an associated monomial $x_F$ in $R$ where $x_F=x_{i_1}\dots x_{i_r}.$ The monomial ideal generated by all monomials $x_F$ such that $F$ is not a face of $\Delta$ is called \textit{Stanley-Reisner ideal} and it is denoted by $I_{\Delta}$. The \textit{face ring} of $\Delta$, denoted by $K[\Delta]$, is defined to be the quotient ring $R/I_{\Delta}$. From \cite[Corollary 6.3.5]{Villareal}, if $\Delta$ is a simplicial complex on $[n]$ of dimension $d$, then $\dim K[\Delta]=d+1=\max\{s: x_{i_1}\dots x_{i_s}\notin I_{\Delta},i_1<\dots<i_s\}$. 
    We recall a nice combinatorial interpretation of the $h$-vector of a shellable simplicial complex, due to McMullen and Walkup (see \cite[Corollary 5.1.14]{Bruns_Herzog}).
	
	\begin{prop}\label{Thm:McMullen-Walkup}
		Let $\Delta$ be a shellable simplicial complex of dimension $d$ with shelling $F_1,\dots,F_m$. For $j\in \{2,\dots,m\}$ we denote by $r_j$ the number of facets of $\langle F_1,\dots,F_{j-1}\rangle\cap \langle F_j\rangle$ and we set $r_1=0$. Let $(h_0,\dots,h_{d+1})$ be the $h$-vector of $K[\Delta]$. Then $h_i=|\{j:r_j=i\}|$ for all $i\in [d+1]$. In particular, up to their order, the numbers $r_j$ do not depend on the particular shelling. 
	\end{prop} 
    \noindent We conclude this section by introducing the main object of study in this paper. Let $\cP$ be a polyomino satisfying the conditions of Theorem \ref{Theorem: Qureshi condition to have Groebner basis}. Then $G(\cP)$ forms the reduced Gr\"obner basis of $I_{\cP}$ with respect to $<$, in particular $\lt_{<}(I_\cP)$ is squarefree and it is generated in degree two. We denote by $\Delta(\cP)$ the simplicial complex on $V(\cP)$ with $\lt_<(I_{\cP})$ as Stanley-Reisner ideal and we call it the \textit{simplicial complex attached to $\cP$}.

	\section{Shellability of the simplicial complex of a frame polyomino}

    \noindent In this section we introduce a new class of non-simple polyominoes, which we call \textit{frame polyominoes}, and we study the shellability of the attached simplicial complex. A frame polyomino is basically a polyomino that can be obtained by removing a parallelogram polyomino from a rectangular polyomino. \\
    Let us start recalling the definition of parallelogram polyomino given in \cite{Parallelogram Hilbert series}. Let $(a,b)\in \ZZ^2$. The sets $\{(a,b),(a+1,b)\}$ and $\{(a,b),(a,b+1)\}$ are called respectively \textit{east step} and \textit{north step} in $\ZZ^2$. A sequence of vertices $(a_0, b_0), (a_1,b_1),\dots,(a_k,b_k)$ in $\ZZ^2$ is called a \textit{north-east path} if $\{(a_i,b_i), (a_{i+1},b_{i+1})\}$ is either an east or a north step. The vertices $(a_0,b_0)$ and $(a_k, b_k)$ are called the \textit{endpoints} of $S$.
	Let $S_1 : (a_0, b_0), (a_1, b_1), \dots , (a_k, b_k)$ and $S_2 : (c_0, d_0), (c_1, d_1),\dots, (c_k, d_k)$ be two north-east
	paths such that $(a_0, b_0) = (c_0, d_0)$ and $(a_k, b_k ) = (c_k, d_k )$. If for all $1 \leq 
	i$ and $j \leq k - 1$, we have $b_i > d_j$ when $a_i = c_j$, then $S_1$ is said to lie above $S_2$.\\
    If $S_1$ lies above $S_2$, we define 	\textit{parallelogram} polyomino, determined by $S_1$ and $S_2$, the set of cells in the region of $\ZZ^2$ bounded above by $S_1$ and below by $S_2$.

	\begin{defn}\rm
		Let $I=[(1,1),(m,n)]$ be an interval of $\ZZ^2$ and $\cS$ be a parallelogram polyomino determined by $S_1 : (a_0, b_0), (a_1, b_1), \dots , (a_k, b_k)$ and $S_2 : (c_0, d_0), (c_1, d_1),\dots, (c_k, d_k)$, where $1<a_0<a_k<m$ and $1<b_0<b_k<n$. We call \textit{frame polyomino} the non-simple polyomino obtained by removing the cells of $\cS$ from $\cP_{I}$.
	\end{defn}
	\noindent In Figure \ref{Figure: Frame polyominoes} we show three examples of frame polyominoes. 
	\begin{figure}[h!]
		\centering
		\subfloat{\includegraphics[scale=0.65]{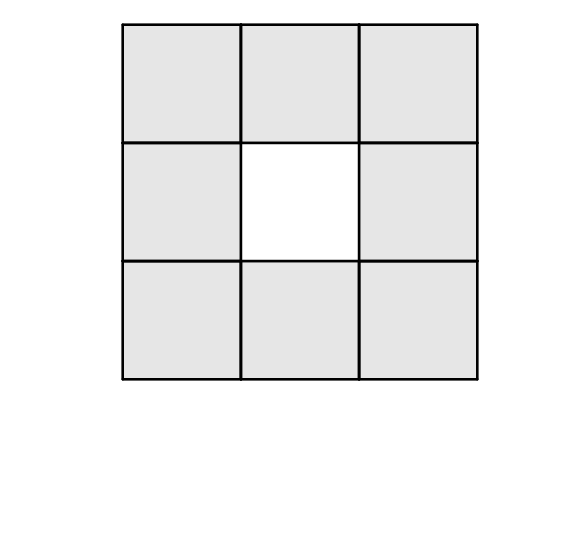}}
		\qquad\qquad
		\subfloat{\includegraphics[scale=0.5]{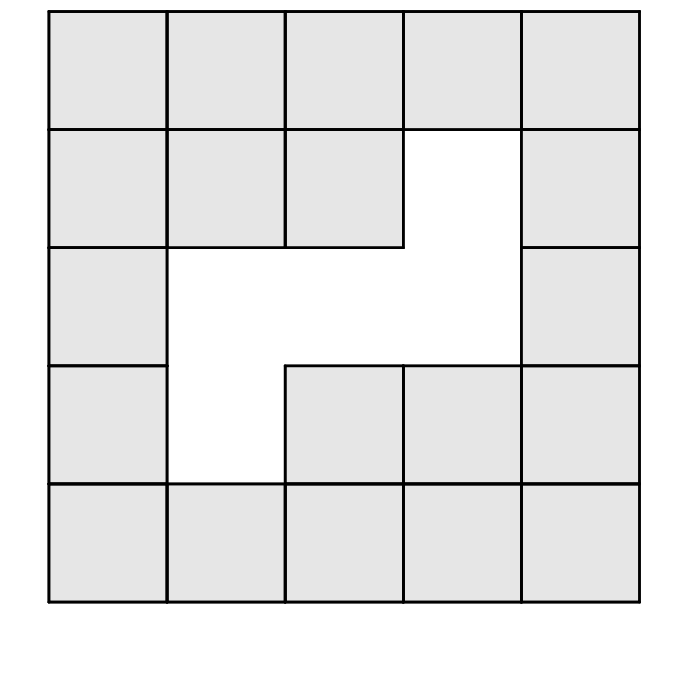}}
  \qquad\qquad
		\subfloat{\includegraphics[scale=0.5]{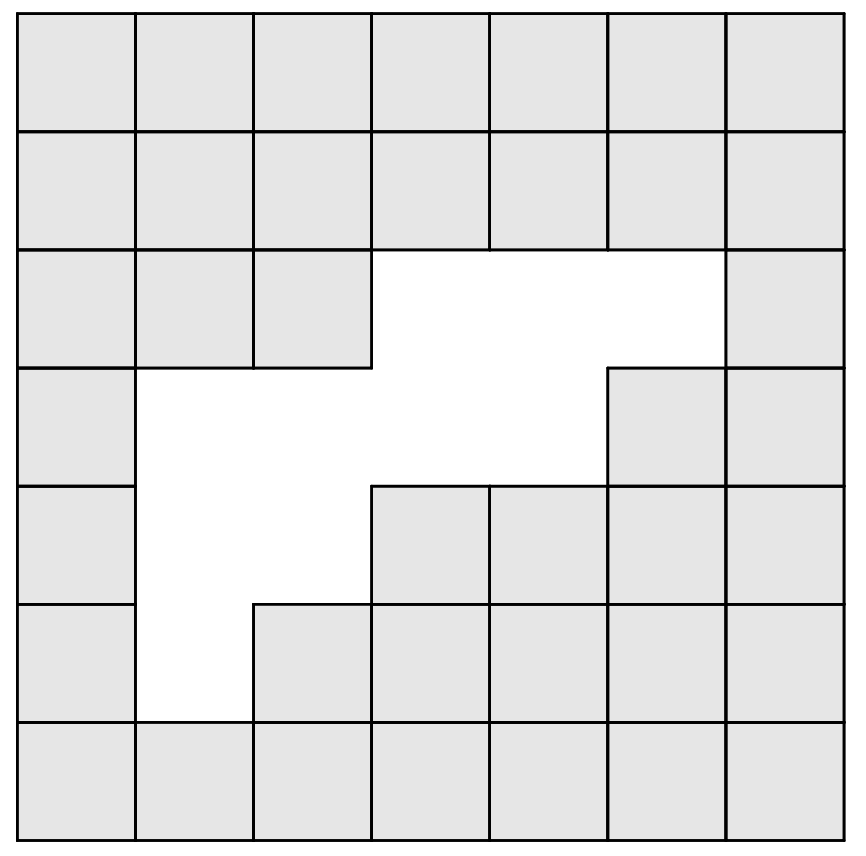}}
		\caption{Examples of frame polyominoes.}
		\label{Figure: Frame polyominoes}
	\end{figure}

        \noindent For a frame polyomino $\cP$ we provide an elementary decomposition, which we use along the paper. It consists of two suitable parallelogram sub-polyominoes, denoted by $\cP_1$ and $\cP_2$. Referring to Figure \ref{Figure: proof for equivalance rook}, $\cP_1$ is the sub-polyomino of $\cP$ highlighted with a red color, and $\cP_2$ is the other one with a hatching filling. Observe that $\cP_1\cap \cP_2=\mathcal{Q}$, with $\cQ=\cP_{[(1,1),(a_0,b_0)]}\cup \cP_{[(a_k,b_k),(m,n)]}$, where $\cP_{[(1,1),(a_0,b_0)]}$ and $\cP_{[(a_k,b_k),(m,n)]}$ are the cell intervals attached respectively to $[(1,1),(a_0,b_0)]$ and $[(a_k,b_k),(m,n)]$. \
	
	\begin{figure}[h!]
		\centering
		\includegraphics[scale=0.6]{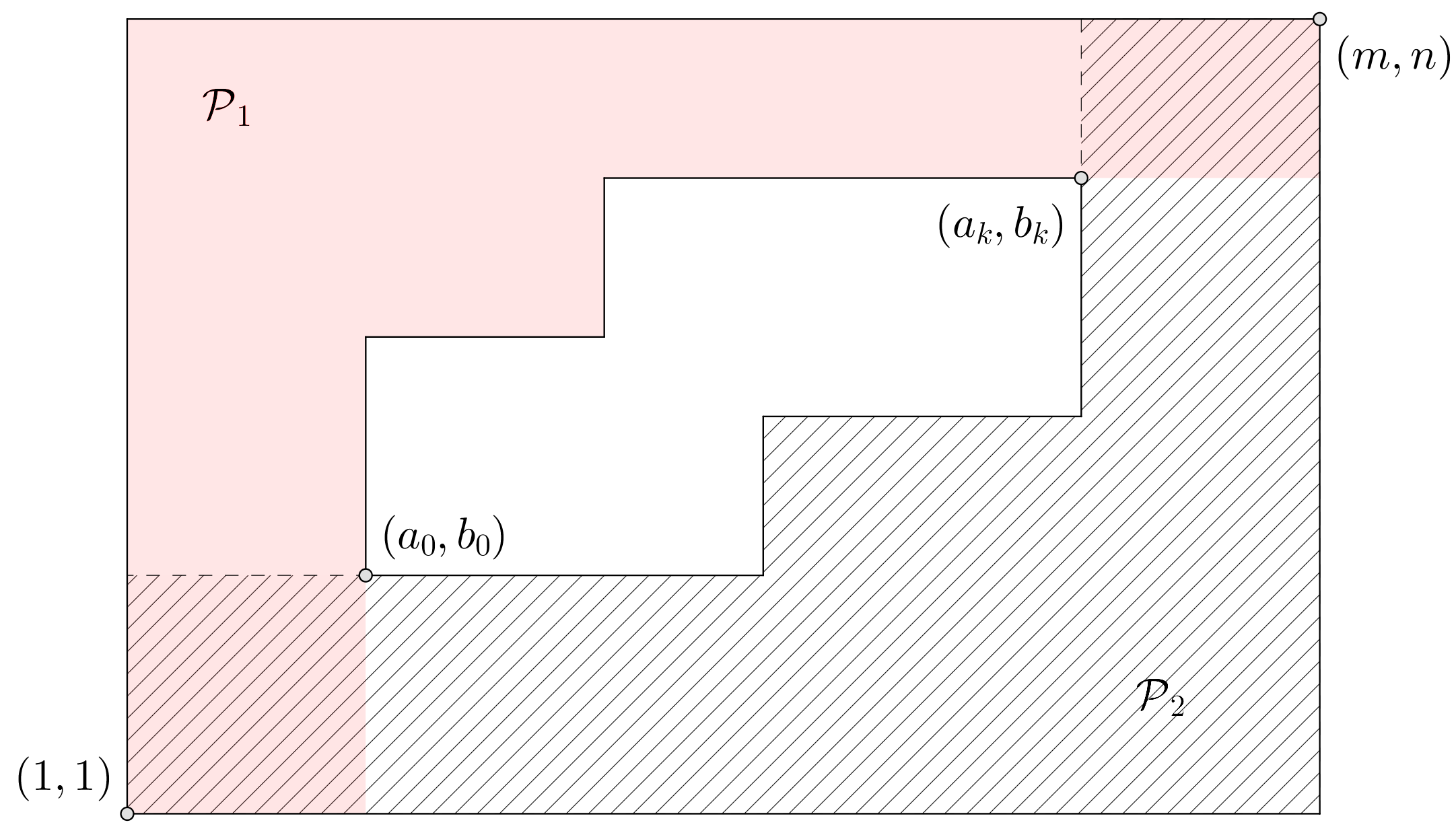}
		\caption{Elementary decomposition of $\cP$.}
		\label{Figure: proof for equivalance rook}
	\end{figure}

    \noindent In the following proposition we show some basic algebraic properties of the polyomino ideal of a frame polyomino. In particular, we determine the Krull dimension of the related coordinate ring using the simplicial complex theory. 
 
	\begin{prop}\label{Prop: Grobner basis of frame polyomino and C.M.}
		Let $\cP$ be a frame polyomino defined by $I=[(1,1),(m,n)]$ and by a parallelogram polyomino $\cS$ determined by $S_1$ and $S_2$ with endpoints $(a_0,b_0)$ and $(a_k,b_k)$. Then:
		\begin{enumerate}
			\item\label{rank1} the set $G(\cP)$ forms the reduced Gr\"{o}bner basis of $I_{\cP}$ with respect to $<$;  
			\item\label{rank2} the initial ideal $\lt_{<}(I_{\cP})$ is generated by the monomials $x_cx_d$ where $c,d$ are the anti-diagonal corners of an inner interval $[a,b]$ of $\cP$;    
			\item\label{rank3} $K[\cP]$ is a normal Cohen-Macaulay domain of Krull dimension $\vert V(\cP)\vert-\rank(\cP)$.
		\end{enumerate}
	\end{prop}
	
	\begin{proof}
		(1) It follows arguing as done in \cite[Corollary 1.2]{Not simple with localization}.\\
		(2) It is an immediate consequence of (1).\\
		(3) It is known that $I_{\cP}$ is a toric ideal from \cite[Theorem 2.1]{Not simple with localization}. By \cite[Corollary 4.26]{binomial ideals} and \cite[Corollary 1.2]{Not simple with localization} we have that $K[\cP]$ is normal and, as a consequence, by \cite[Theorem 6.3.5]{Bruns_Herzog} we obtain that $K[\cP]$ is Cohen-Macaulay. Now we compute the Krull dimension of $K[\cP]$. Consider the simplicial complex $\Delta(\cP)$ attached to $\cP$.  
        Observe that $\Delta(\cP)$ is a shellable simplicial complex by \cite[Theorem 9.6.1]{Villareal}. Moreover, we have that $\vert V(\cP)\vert = \vert V(\cP_I)\vert -\vert V(\cS)\vert + \vert S_1\vert +\vert S_2\vert -2$ and $\rank(\cP)=\rank(\cP_I)-\rank(\cS)$. Hence:
		\begin{equation}
			\vert V(\cP)\vert -\rank(\cP)= \vert V(\cP_I)\vert- \rank(\cP_I) - (\vert V(\cS)\vert-\rank(\cS)) + \vert S_1\vert +\vert S_2\vert -2.
		\end{equation} 
		\noindent Observe that $\cP_I$ and $\cS$ satisfy the conditions in Theorem \ref{Theorem: Qureshi condition to have Groebner basis}, so we denote by $\Delta(\cP_I)$ and $\Delta(\cS)$ the simplicial complexes attached to $\cP_I$ and $\cS$ respectively. 
        In particular we have that $\dim K[\cP_I]=\dim K[\Delta(\cP_I)]$ and $\dim K[\cS]=\dim K[\Delta(\cS)]$. Since $\cP_I$ and $\cS$ are simple polyominoes, from \cite[Theorem 2.1]{Simple equivalent balanced} and \cite[Corollary 3.3]{def balanced} we know that $K[\cP_I]$ and $K[\cS]$ are normal Cohen-Macaulay domain with respectively $\dim K[\cP_I]=\vert V(\cP_I)\vert -\rank(\cP_I)$ and $\dim K[\cS]=\vert V(\cS)\vert -\rank(\cS)$. As a consequence $\Delta(\cP_I)$ and $\Delta(\cS)$ are pure, so $\dim(\Delta(\cP_I))=\vert F_I\vert$ and $\dim(\Delta(\cS))=\vert S_1\vert=\vert S_2\vert$, where $F_I=[(1,1),(1,n)]\cup  [(1,n),(m,n)]$. Set $S^*=S_2\backslash \{(a_0,b_0),(a_k,b_k)\}$. 
		Therefore, from $(\ref{rank1})$ and from the previous arguments, we have that
		$$ \vert V(\cP)\vert -\rank(\cP)= \vert F_I\vert - \vert S_1\vert + \vert S_1\vert +\vert S_2\vert -2=\vert F_I\vert+\vert S^*\vert=\vert F_I\sqcup S^*\vert .$$
	We prove that $F_I\sqcup S^*$ is a facet of $\Delta(\cP)$. Firstly observe that $F_I\sqcup S^*$ is a face of $\Delta(\cP)$ because there does not exist any inner interval of $\cP$ whose anti-diagonal corners are in $F_I\sqcup S^*$. Due to the maximality, suppose by contradiction that there exists a face $K$ of $\Delta(\cP)$ such that $F_I\sqcup S^*\subset K$. Let $w\in K\setminus (F_I\sqcup S^*)$. If $w\in V(\cP_1)\setminus F_I$ then the interval having $(1,n)$ and $w$ as anti-diagonal corners is an inner interval of $\cP$, which is a contradiction with $(\ref{rank2})$. If $w\in V(\cP_2)\setminus (V(\cP_1)\cup S^*)$, then it is easy to see that there is an inner interval of $\cP$ whose anti-diagonal corners are $w$ and a vertex in $\{(1,b_0),(a_k,n)\}\sqcup S^*$, that is a contradiction with $(\ref{rank2})$. Hence $F_I\sqcup S^*$ is a facet of $\Delta(\cP)$, so we get the desired conclusion.
	\end{proof}

	\noindent As seen in (3) of Proposition \ref{Prop: Grobner basis of frame polyomino and C.M.}, the simplicial complex $\Delta(\cP)$ attached to a frame polyomino $\cP$ is shellable. In order to define a suitable shelling order of $\Delta(\cP)$, we introduce the notion of a \textit{step} of a face of $\Delta(\cP)$. 
	
	\begin{defn}\rm\label{Definition: Step of a facet}
	\noindent Let $\cP$ be a polyomino satisfying Theorem \ref{Theorem: Qureshi condition to have Groebner basis} and $\Delta(\cP)$ be the simplicial complex attached to $\cP$. Let $F$ be a face of $\Delta(\cP)$ with $\vert F\vert \geq 3$ and $F'=\{(a,b),(c,b),(c,d)\}\subseteq F$. We say that $F'$ forms a \textit{step} in $F$ or that $F$ has a \textit{step} $F'$ if:
	\begin{enumerate}
			\item $a<c$ and $b<d$;
			\item for every integer $i\in \{a+1,\dots,c-1\}$ there does not exist $(i,b)$ in $F$;
			\item for every integer $j\in \{b+1,\dots,d-1\}$ there does not exist $(c,j)$ in $F$;
			\item $(c,b)$ is the lower right corner of a cell of $\cP$.
	\end{enumerate}
	In such a case the vertex $(c,b)$ is said to be the \textit{lower right corner} of $F'$. 
	\end{defn}

	\begin{exa}\rm 
	Let $\cP$ be the polyomino in Figure \ref{Figure: Examples steps}. It is trivial to see that $\cP$ satisfies Theorem \ref{Theorem: Qureshi condition to have Groebner basis}, so we consider the simplicial complex $\Delta(\cP)$ attached to $\cP$. The blue vertices represent a facet $F$ of $\Delta(\cP)$. $F$ has five steps, which are $\{a_1,a_2,a_7\}$, $\{a_2,a_3,a_4\}$, $\{a_7,a_8,a_{11}\}$, $\{a_{12},a_{13},a_{15}\}$ and $\{a_9,a_{10},a_{14}\}$. Note that $\{a_5,a_6,a_9\}$ is not a step of such a facet because $a_6$ is not the lower right corner of a cell of $\cP$. 
	\begin{figure}[h!]
		\centering
		\includegraphics[scale=1]{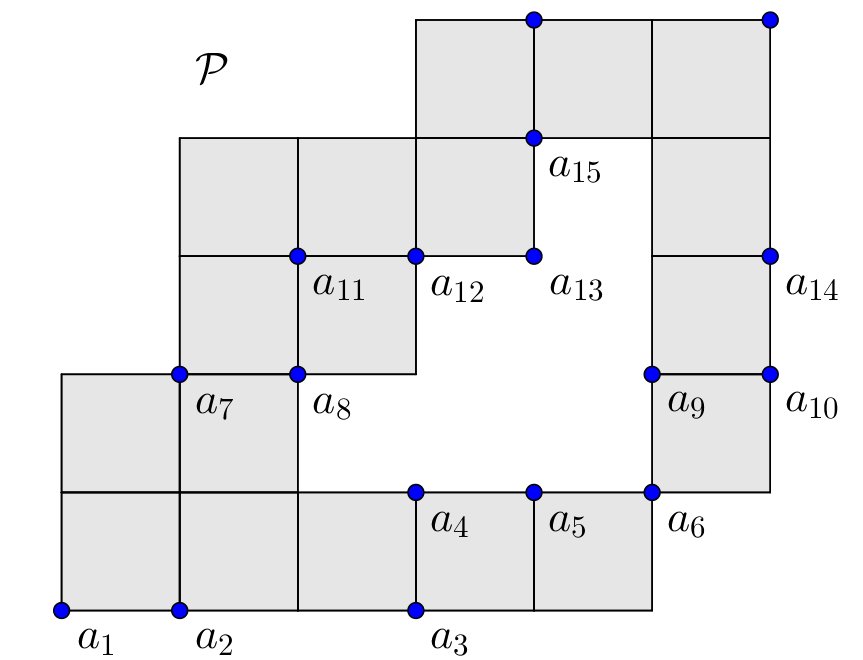}
		\caption{Example of steps in a facet of $\Delta(\cP)$.}
		\label{Figure: Examples steps}
	\end{figure}
	\end{exa}

	\noindent We show a useful property of a step of a facet of the simplicial complex attached to a frame polyomino. 

	\begin{lemma}\label{Lemma: a step gives an inner interval}
	Let $\cP$ be a frame polyomino. Let $\Delta(\cP)$ be the simplicial complex attached to $\cP$ and $F'=\{(a,j),(i,j),(i,b)\}$ be a step of a facet of $\Delta(\cP)$. Then $[(a,j),(i,b)]$ is an inner interval of $\cP$.  	
	\end{lemma}

	\begin{proof}
	If $a=i-1$ and $b=j+1$ then we have the conclusion immediately from (4) of Definition \ref{Definition: Step of a facet}. Assume that $a\neq i-1$ or $b\neq j+1$. If $(i,j)\in V(\cP_1)$ then $[(a,j),(i,b)]$ is an inner interval of $\cP$, by the structure of $\cP$ and by Definition \ref{Definition: Step of a facet}. Assume that $(i,j)\notin V(\cP_1)$, so $(i,j)\in V(\cP_2)\setminus (\cP_{[(1,1),(a_0,b_0)]}\cup \cP_{[(a_k,b_k),(m,n)]})$. Suppose by contradiction that $[(a,j),(i,b)]$ is not an inner interval of $\cP$. Then there exists a cell $C$ not belonging to $\cP$ with lower right corner $(h,l)$ such that $(h,l)\neq (i,j)$ and $a<h\leq i$, $j\leq l<b$. Observe that all vertices of $\cP$ in $[(a+1,1),(m,j)]\setminus \{(i,j)\}$ and $[(i,1),(m,b-1)]\setminus\{(i,j)\}$ do not belong to $F$. Suppose $h=i$. Then there does not exist any inner interval of $\cP$ having $(h,l)$ and another vertex in $F$ as anti-diagonal corners, so $(h,l)\in F$ due to the maximality of $F$, but this is a contradiction with (3) of Definition \ref{Definition: Step of a facet}. A similar contradiction arises if $ l=j$. Therefore $h\neq i$ and $l\neq j$. It is not restrictive to assume that $l$ is the minimum integer such that the cell $C$ with lower right corner $(h,l)$ belongs to $[(a,j),(i,b)]$ but not to $\cP$. Let $J$ be the maximal inner interval of $\cP$ having $(i,j)$ as the lower right corner and containing $(a,j)$; moreover, we denote by $H$ the maximal edge interval of $\cP$ containing $(a,j)$ and $(i,j)$. Note that no vertex of $J\setminus H$ belongs to $F$. Therefore, as explained before, $(h,j)\in F$, so we get again a contradiction with (2) of Definition \ref{Definition: Step of a facet}. In conclusion $[(a,j),(i,b)]$ is an inner interval of $\cP$.
	\end{proof}

	\begin{rmk}\rm 
 	The assumption that $\cP$ is a frame polyomino is important for the claim of Lemma \ref{Lemma: a step gives an inner interval}. In fact, if $\cP$ is the polyomino in Figure \ref{Figure: closed path plus cell} then $\cP$ satisfies the condition of Theorem \ref{Theorem: Qureshi condition to have Groebner basis} and the set of the orange vertices determines a facet $F$ of the simplicial complex $\Delta(\cP)$ attached to $\cP$, where $\{v_1,v_2,v_3\}$ is a step of $F$ but $[v_1,v_3]$ is not an inner interval of $\cP$.      
 	
 	\begin{figure}[h!]
 		\centering
 		\includegraphics[scale=0.9]{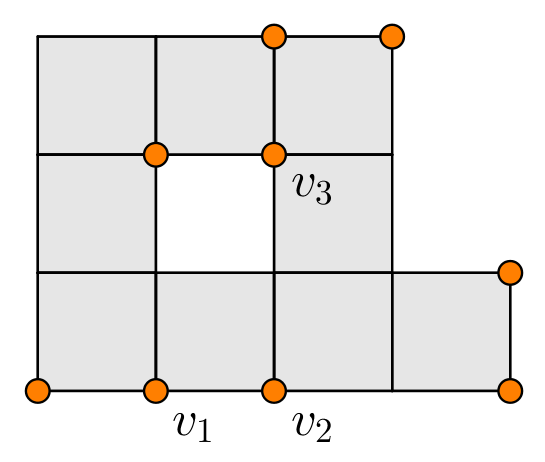}
 		\caption{A facet.}
 		\label{Figure: closed path plus cell}
 	\end{figure}
	\end{rmk}

		\noindent To give a complete description of the structure of a facet of the simplicial complex associated with a frame polyomino, we start introducing some basics on a finite distributive lattice, following \cite{Parallelogram Hilbert series}. Let $P$ be a poset with partial order relation $\prec$. A \textit{chain} of $P$ is a totally ordered subset of $P$. A chain $C$ is said to be \textit{maximal} if there does not exist any chain of $P$ containing $C$. Given $a \in P$, the \textit{rank} of $a$ in $P$ is the supremum of the length of chains in $P$ that descends from $a$. The \textit{rank} of $P$ is the supremum of the length of chains of $P$ and it is denoted by $\rank(P)$. Let $L$ be a simple planar distributive lattice. We say that, given $x,y \in L$, $y$ covers $x$ if $x \prec y$ and there is no $z\in L$ such that $x \prec z \prec y$; in such a case we use the notation $x\rightarrow y$. \\
        Along the paper, as done in \cite{Parallelogram Hilbert series}, if a polyomino $\cP$ has a structure of a distributive lattice on $V(\cP)$, then with abuse of notation we refer to $\cP$ as a distributive lattice. From \cite[Proposition 2.3]{Parallelogram Hilbert series} we know that a finite collection of cells $\cP$ is a parallelogram polyomino if and only if	$\cP$ is a simple planar distributive lattice.\\
        Let $\cP$ be a parallelogram polyomino. Let $\rank(\cP) = d + 1$ as distributive lattice, $ \mathfrak{m} : \min L = x_0 \rightarrow x_1 \rightarrow x_2 \rightarrow
	\dots x_{d+1} = \max L$ be a maximal chain of $\cP$ and $C = [(i, j), (i+1, j+1)]$ be a cell of $\cP$. We say that $\mathfrak{m}$ has a \textit{descent} at $C$ if $\mathfrak{m}$ passes
	through the edges $(i, j) \rightarrow (i + 1, j)$ and $(i + 1, j) \rightarrow(i + 1, j + 1)$.
 
\begin{rmk}\rm \label{Remark: facets in a parallelogram}
	Let $\cP$ be a parallelogram polyomino. Observe that $\cP$ satisfies the conditions in Theorem \ref*{Theorem: Qureshi condition to have Groebner basis}, so the set of generators of $I_{\cP}$ forms the (quadratic) reduced Gr\"obner basis of $I_{\cP}$ with respect to $<$. Let $\Delta(\cP)$ be the simplicial complex attached to $\cP$. For all $j\geq 0$, it is easy to see that every maximal chain of $\cP$ with $j$ descents as a distributive lattice is a facet of $\Delta(\cP)$ with $j$ steps, and vice versa. 
	\end{rmk}

    \begin{lemma}\label{Lemma: on the max. edge int. exists a point in F}
	Let $\cP$ be a frame polyomino defined by $I=[(1,1),(m,n)]$ and by a parallelogram polyomino $\cS$ determined by the north-east paths $S_1 : (a_0, b_0), (a_1, b_1), \dots , (a_k, b_k)$ and $S_2 : (c_0, d_0), (c_1, d_1),\dots, (c_k, d_k)$ with $(a_0, b_0) = (c_0, d_0)$ and $(a_k, b_k ) = (c_k, d_k )$. 
 Let $\Delta(\cP)$ be the simplicial complex attached to $\cP$ and $F$ be a facet of $\Delta(\cP)$. Then for all maximal edge intervals $\cL$ of $\cP$ there exists $v\in \cL$ belonging to $F$.  
	\end{lemma}
	
	\begin{proof}
    Let $V$ be a maximal vertical edge interval of $\cP$. We prove that there exists $v$ in $V$ belonging to $F$. We distinguish the following four cases.
    \begin{itemize}
        \item[Case 1] Assume that $V=[(a,1),(a,n)]$ with $a\in \{1,\dots,a_0\}$. Suppose that $a=1$. Then $(1,1)\in F$, since $(1,1)$ is not the anti-diagonal corner of an inner interval of $\cP$, and trivially $(1,1)\in V$. Therefore we get the claim for $a=1$. Suppose now that $1<a\leq a_0$. Consider $G=\{(i,j)\in F:1\leq i<a,1\leq j\leq n \}$. Observe that $G\neq \emptyset$ since $(1,1)\in G$. Let $(i_1,k_1)\in G$ with $k_1=\max \{j:(i,j)\in G\}$. We want to show that $(a,k_1)\in F$. First of all, we observe that for every inner interval $\cI$ of $\cP$ having $(i_1,k_1)$ as anti-diagonal corner the other anti-diagonal corner of $\cI$ does not belong to $F$, otherwise there exists an inner interval of $\cP$ whose two anti-diagonal corners are in $F$, which is a contradiction with $(2)$ of Proposition \ref{Prop: Grobner basis of frame polyomino and C.M.}. Moreover, the vertices in $V_1=\{(i,j)\in V(\cP):1\leq i< a,k_1<j\leq n\}$ are not in $F$ due to the maximality of $k_1$. In order to prove that $(a,k_1)\in F$, it is sufficient to prove that, for every inner interval of $\cP$ with anti-diagonal corner $(a,k_1)$, the other anti-diagonal corner is not in $F$; in fact, due the maximality of $F$ it follows necessarily that $(a,k_1)\in F$. Let $\cK$ be an inner interval of $\cP$ having $(a,k_1)$ as anti-diagonal corner and $v=(r,s)$ be the other anti-diagonal corner of $\cK$. We have just two cases to examine. If $r<a$ and $s>k_1$ then $v\in V_1$, so $v\notin F$. If $r>a$ and $s<k_1$, then we show that $v\notin F$. In fact, suppose by contradiction that $v\in F$. Let $\tilde{\cK}$ be the interval of $\Z^2$ with anti-diagonal corners $(i_1,k_1)$ and $v$. Denote by $\cC$ the interval of $\Z^2$ having $(i_1,k_1)$ and $(a,s)$ as anti-diagonal corners. Note that $\cC$ is an inner interval of $\cP$ due to the structure of $\cP$ as $a\leq a_0$, and that $\tilde{\cK}=\cK\cup \cC$. Since $\cK$ and $\cC$ are inner intervals of $\cP$, then $\tilde{\cK}$ is an inner interval of $\cP$. This is a contradiction because $(i_1,k_1)$ and $v$ are anti-diagonal corners of $\tilde{\cK}$ and they belong to $F$ at the same time. Hence $v$ cannot be in $F$. In conclusion, we get the desired claim.  
        \item[Case 2] Assume that $V=[(a,1),(a,d_h)]$, where $a_0<a< a_k$ and $(a,d_h)\in S_2$ for a suitable $h\in [k-1]$. Suppose that $a=a_0+1$. Let $k_2=\max \{j:(i,j) \in F\cap [(1,1),(a_0,b_0)]\}$. Then it is easy to see for every inner interval of $\cP$ with anti-diagonal corner $(a_0+1, k_2)$, the other anti-diagonal corner is not in $F$, so $(a_0+1,k_2) \in F$. Suppose that $a_0+1<a\leq a_k$. Let $S_2'\subset S_2$ be the north-east path $(a_0+1,d_1),\dots,(a,d_h)$. We denote by $\cR$ the parallelogram polyomino determined by the two north-east paths $[(a_0+1,1),(a_0+1,d_1)]\cup S_2'$ and $[(a_0+1,1),(a,1)]\cup[(a,1),(a,d_h)]$. Let $G'= V(\cR)\cap F$. $G'\neq \emptyset$ because $(a_0+1,k_2)\in G'$. Set $k_3=\max \{j:(i,j) \in G'\}$. Using similar arguments as done in Case 1, we prove that $(a,k_3)\in F$, so the claim follows.
        \item[Case 3] The case when $V=[(a,b_j),(a,n)]$, where $a_0< a< a_k$ and $(a,b_j)\in S_1$ for an opportune $h\in [k-1]$, can be proved similarly as done in Case 2.
        \item[Case 4] Assume that $V=[(a,1),(a,n)]$, where $a_k\leq a\leq m$. If $a=a_k$, then we set $k_4=\max\{j\in [b_k]:(a_k-1,j)\in F\}$ and, using the arguments explained in Case 2, it is easy to show that $(a_k,k_4)\in F$. If $a_k< a\leq m$, then we get the claim arguing as done in Cases 1 and 2.
    \end{itemize} 

    \noindent For a maximal horizontal edge interval of $\cP$, the claim can be shown as done before. Therefore the Lemma is completely proven.   	    
	\end{proof}

    \begin{disc}\rm \label{Discussion: facet}
    Let $\cP$ be a frame polyomino and $F$ be a facet of the simplicial complex $\Delta(\cP)$ attached to $\cP$. From Lemma \ref{Lemma: on the max. edge int. exists a point in F} we know that in every maximal edge interval of $\cP$ we can find an element of $F$. Now, we want to describe how the elements of $F$ are arranged in $\cP$.\\ 
    Let $v=(i,j)$ be the maximal vertex of $F$ in $V(\cP_{[(1,1),(a_0,b_0)]})$ with respect to $<$ and $w=(t,l)$ be the minimal vertex of $F$ in $V(\cP_{[(a_k,b_k),(m,n)]})$ with respect to $<$. Observe that $v$ and $w$ are unique. In fact, if $v=(i',j')$ is another maximal vertex of $F$ in $V(\cP_{[(1,1),(a_0,b_0)]})$ with respect to $<$, then either $i'<i$ and $j'>j$ or $i'>i$ and $j'<j$, so there exists an inner interval of $\cP$ with $v,v'$ as anti-diagonal corners.  But this is a contradiction with $(2)$ of Proposition \ref{Prop: Grobner basis of frame polyomino and C.M.}, since $v,v'\in F$. The same argument follows for $w$. Moreover, we point out that $(1,1),(m,n)\in F$, because they cannot be the  anti-diagonal corners of an inner interval of $\cP$. \\
    Consider $\cP_{[(1,1),(a_0,b_0)]}$ and $v=(i,j)$. We examine the following four cases.
    \begin{enumerate}
        \item If $i=1$ and $j=1$ then no vertex from $[(1,1),(a_0,b_0)]$ belongs to $F$.
        \item  If  $i>1$ and $j=1$, then $[(1,1),(i,1)]\subset F$, because no vertex from $\{(p,q)\in V(\cP):p<i,q>1\}$ can be in $F$ and due to the maximality of $F$.
        \item If $i=1$ and $j>1$, then $[(1,1),(1,j)]\subset F$, arguing as in $(1)$.
        \item Assume that $i>1$ and $j>1$. Consider $(1,2)$ and $(2,1)$ and we show that either $(1,2)$ or $(2,1)$ belongs to $F$. First of all, both of them cannot be in $F$, otherwise, we have a contradiction with (2) of Proposition \ref{Prop: Grobner basis of frame polyomino and C.M.}. Moreover, since the vertices in $\{(p,q)\in V(\cP):p<i,q>j\}\cup \{(p,q)\in V(\cP):p>i,q<j\}$ are not in $F$ and due to the maximality of $F$, at least one of $(1,2)$ and $(2,1)$ belongs to $F$. We may assume that $(2,1)\in F$, because the arguments for $(1,2)$ are the same. We have two sub-cases. If $i=2$, that is $v=(2,j)$, then the only vertices of $F$ in $V(\cP_{[(1,1),(a_0,b_0)]})$ are given by $\{(1,1),(2,1)\}\cup [(2,1),(2,j)]$. If $i>2$, then the cell having $(2,1)$ as lower left corner is in $\cP_{[(1,1),(a_0,b_0)]}$, so $(2,2),(3,1)\in V(\cP_{[(1,1),(a_0,b_0)]})$. Arguing as before, we have that either $(2,2)$ or $(3,1)$ belongs to $F$, and finally we iterate that procedure until $v$. 
    \end{enumerate}
    \noindent  Hence the elements of $F$ form a chain $\mathfrak{c}_1:(1,1)\rightarrow\dots \rightarrow v=(i,j)$ in $\cP_{[(1,1),(a_0,b_0)]}$.\\
    Now, we focus on $\cP_1\setminus \cQ$. Observe that no vertex in $\{(p,q)\in V(\cP):p<i,q>j\}\cup \{(p,q)\in V(\cP):p>i,q<j\}$ belongs to $F$ because $v=(i,j)\in F$. Moreover, due to the maximality of $v$ in $V(\cP_{[(1,1),(a_0,b_0)]})$ with respect to $<$, we have that the vertices in $\{(p,q)\in V(\cP):p>i,j\leq q<b_0+1\}$ do not belong to $F$. Hence for every inner interval of $\cP$ with anti-diagonal corner $(i, b_0+1)$, the other anti-diagonal corner is not in $F$, so $(i,b_0+1)\in F$ because of the maximality of $F$. Starting from $(i,b_0+1)$, we argue similarly as done before in $\cP_{[(1,1),(a_0,b_0)]}$ and we continue that procedure until $(a_k-1,l)$. Therefore, the elements of $F$ form a chain $\mathfrak{c}_2:(i,b_0+1)\rightarrow\dots \rightarrow (a_k-1,l)$ in $\cP_1\setminus \cQ$.\\
    By similar arguments, the elements of $F$ provide a chain $\mathfrak{c}_3:(a_0+1,j)\rightarrow\dots \rightarrow (t,b_k-1)$ in $\cP_2\setminus \cQ$ and another one $\mathfrak{c}_4:w=(t,l)\rightarrow\dots \rightarrow (m,n)$ in $\cP_{[(a_k,b_k),(m,n)]}$. \\
    Therefore $F$ is described by the chains $\mathfrak{c}_1$, $\mathfrak{c}_2$, $\mathfrak{c}_3$ and $\mathfrak{c}_4$. In Figure \ref{Figure: Facets Frame polyominoes} we show two facets of the simplicial complex attached to a frame polyomino. 
   
    \begin{figure}[h!]
		\centering
		\subfloat[]{\includegraphics[scale=0.95]{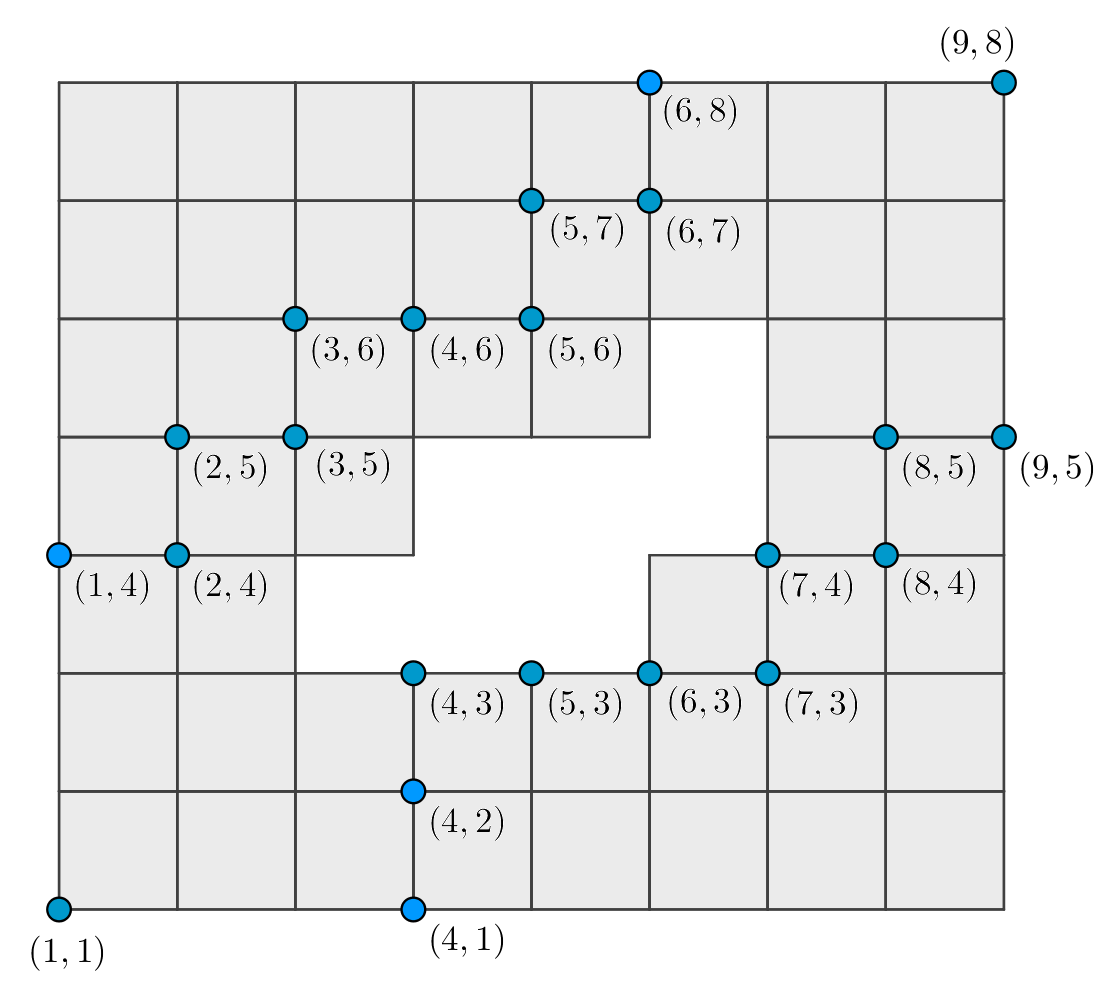}}
		\subfloat[]{\includegraphics[scale=0.95]{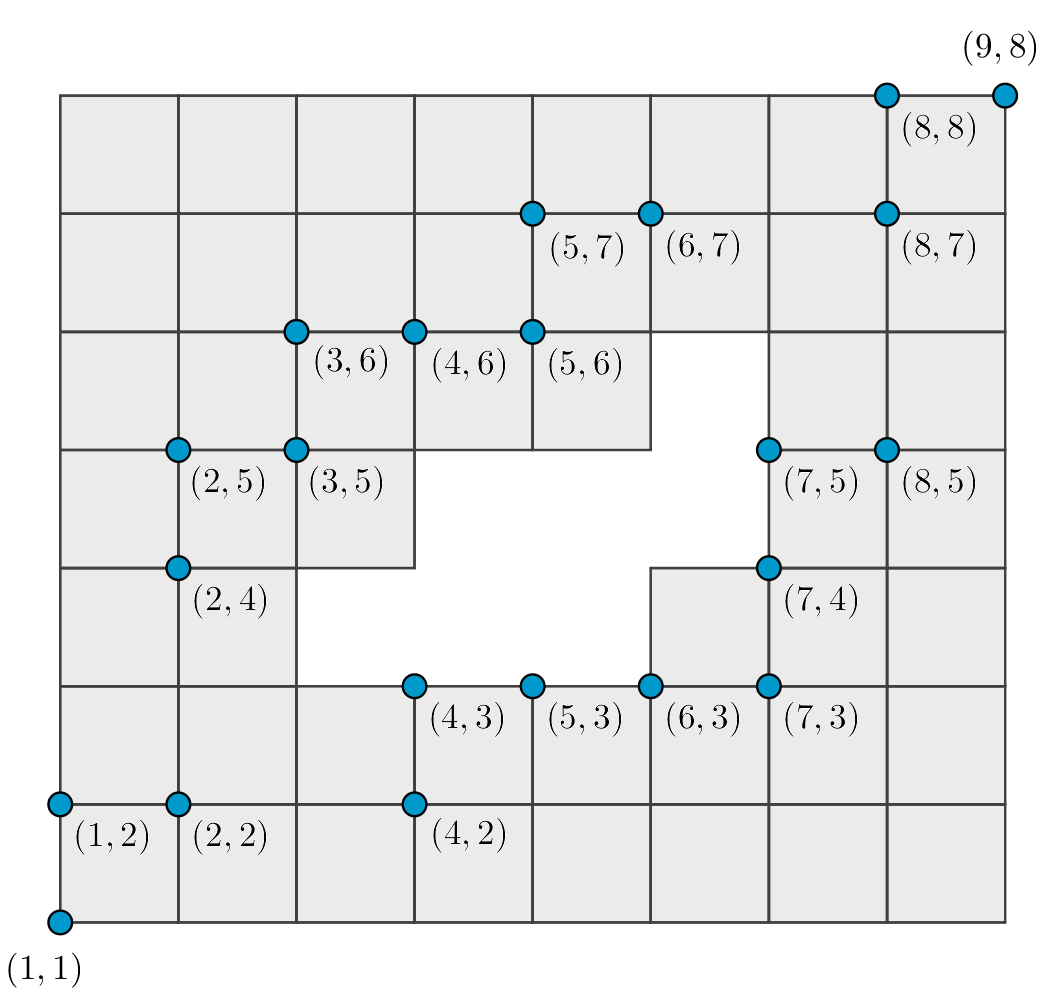}}
		\caption{Examples of facets in a frame polyomino.}
		\label{Figure: Facets Frame polyominoes}
	\end{figure}
    
      \noindent Moreover, if we denote by $n(\mathfrak{c}_i)$ the number of the descents in $\mathfrak{c}_i$ for all $i\in [4]$ and by $n(F)$ the number of the steps in $F$, we invite the reader to observe that $\sum_{i=1}^4 n(\mathfrak{c}_i)\leq n(F)\leq \sum_{i=1}^4 n(\mathfrak{c}_i)+4$. In fact, every descent of $\mathfrak{c}_i$ corresponds to a step of $F$, so $\sum_{i=1}^4 n(\mathfrak{c}_i)\leq n(F)$, and there are at most four steps in $F$ that are not descents of a chain $\mathfrak{c}_i$, as shown in Figure \ref{Figure: Facets Frame polyominoes} (B) by $\{(1,2),(2,2),(2,4)\}$, $\{(2,2),(4,2),(4,3)\}$, $\{(6,7),(8,7),(8,8)\}$ and $\{(7,5),(8,5),(8,7)\}$, so $n(F)\leq \sum_{i=1}^4 n(\mathfrak{c}_i)+4$.
    \end{disc}
    
    \noindent In the following definition we introduce a way to compare two facets of the simplicial complex attached to a polyomino having the property described in Theorem \ref{Theorem: Qureshi condition to have Groebner basis}.
    \begin{defn}\rm 
      Let $\cP$ be a polyomino satisfying Theorem \ref{Theorem: Qureshi condition to have Groebner basis} and $\Delta(\cP)$ be the simplicial complex attached to $\cP$. We denote by $\cF_{\cP}$ the set of the facets of $\Delta(\cP)$ and we define a suitable order on $\cF(\cP)$. First of all, we introduce the following total order on $V(\cP)$. Let $a,b\in V(\cP)$ with $a=(i,j)$ and $b=(k,l)$, we set $b<'a$ if $j > l$, or, $j = l$ and $i > k$. Now, consider two distinct facets $F=\{a_1,\dots, a_d\}$ and $G=\{b_1,\dots, b_d\}$ of $\Delta(\cP)$, where $a_{i+1}<'a_i$ and $b_{i+1}<'b_{i}$ for all $i=1,\dots,d-1$. Let $j$ be the smallest integer in $[d]$ such that $b_j\neq a_j$. Then we define $F<_{\mathrm{lex}} G$ if $a_j<' b_j$. Moreover, if $\cF_{\cP}=\{F_0,F_1,\dots, F_r\}$, then we say that $\cF_{\cP}$ is \textit{lexicographically order in descending} if $F_{i+1}<_{\mathrm{lex}} F_i$ for all $i=0,\dots, r-1$.     
    \end{defn}
   \begin{exa}\rm
       Let $\cP$ be the polyomino in Figure \ref{Figure: Facets Frame polyominoes} and $F$ and $G$ respectively the facets of $\Delta(\cP)$ showed in (A) and (B), that are
       \begin{itemize}
           \item[-]
        $F=\{(9,8),(6,8),(6,7),(5,7),(5,6),(4,6),(3,6),(9,5),(8,5),(3,5),(2,5)(8,4),(7,4),(2,4),\\(1,4),(7,3),(6,3),(5,3),(4,3),(4,2),(4,1),(1,1)\}$,
       \item[-] $G=\{(9,8),(8,8),(8,7),(6,7),(5,7),(5,6),(4,6),(3,6),(8,5),(7,5),(3,5)(2,5),(7,4),(2,4),\\(7,3),(6,3),(5,3),(4,3),(4,2),(2,2),(1,2),(1,1)\}.$
       \end{itemize}
      \noindent Observe that in $F$ and $G$ the first different vertices from the left to right are in the second position and $(6,8)<'(8,8)$, so $F<_{\mathrm{lex}}G$.
   \end{exa}
     \noindent If $\cP$ is frame polyomino, we set $F_0=[(1,1),(1,n)]\cup  [(1,n),(m,n)]\cup \big(S_2\backslash \{(a_0,b_0),(a_k,b_k)\}\big)$, which is a facet of $\Delta(\cP)$ as proved in (3) of Proposition \ref{Prop: Grobner basis of frame polyomino and C.M.}. Moreover, observe from Discussion \ref{Discussion: facet} that $F_0$ is the unique facet in $\Delta(\cP)$ with no step and $F<_{\mathrm{lex}} F_0$ for all $F\in \cF_{\cP}$. Now, we are ready to  prove the main result of this section.    
    
	\begin{thm}\label{Proposition: The facet less the lower right corner generates the intersection}
    Let $\cP$ be a frame polyomino and $\Delta(\cP)$ be the simplicial complex attached to $\cP$. Suppose that $\cF_{\cP}$ is lexicographically ordered in descending and consider a facet $F\neq F_0$ of $\Delta(\cP)$. Set $\cS(F)=\{G\in \cF_{\cP}: F<_{\mathrm{lex}}G\}$ and $\cK_F=\{F\backslash\{v\}:v\ \text{is the lower right corner of a step of}\ F\}$. Then:
    \begin{enumerate}
        \item $\langle \cS(F)\rangle \cap \langle F\rangle=\langle \cK_F\rangle$ and, in particular, $\cF_{\cP}$ forms a shelling order of $\Delta(\cP)$;
        \item the $i$-th coefficient of the $h$-polynomial of $K[\cP]$ is the number of the facets of  $\Delta(\cP)$ having $i$ steps.
    \end{enumerate}
	\end{thm}
	
	\begin{proof}
 	(1) Firstly, we show that $\langle \cK_F\rangle\subseteq \langle \cS(F)\rangle \cap \langle F\rangle$. Let $F\backslash \{(i,j)\}$, where $(i,j)$ is the lower right corner of a step $F'=\{(i,b),(i,j),(a,j)\}$ of $F$. Trivially $F\backslash \{(i,j)\}\subset F$. 
 	We may assume that $(i,j), (i,b)\in V(\cP_2)\setminus (V(\cP_{[(1,1),(a_0,b_0)]})\cup V(\cP_{[(a_k,b_k),(m,n)]}))$ and $(a,j)\in V(\cP_{[(1,1),(a_0,b_0)]})$, as in Figure Figure \ref{Figure: proof for shelling order}, since all other cases can be proved by similar arguments. From Lemma \ref{Lemma: a step gives an inner interval}, $[(a,j),(i,b)]$ is an inner interval of $\cP$. Since $(i,b),(a,j),(i,j)\in F$, observe that no vertex in 
    \begin{align*}
        \mathcal{N}=& \Big([(i,1),(m,b-1)]\setminus\{(i,j)\}\Big)\cup \Big([(a+1,1),(i,j)]\setminus\{(i,j)\}\Big)\cup \\
        &\cup \Big([(1,j+1),(a-1,n)]\Big)\cup \Big([(a,j),(i,b)]\setminus \{(a,j),(i,j),(i,b)\}\Big)
    \end{align*} belongs to $F$. For instance, $\mathcal{N}$ consists of the white vertices in the blue, red, and grey parts in Figure \ref{Figure: proof for shelling order}. We consider the set $H=(F\backslash \{(i,j)\})\cup \{(a,b)\}$ and we prove that $H$ is a facet of $\Delta(\cP)$. In order to show that $H$ is a face of $\Delta(\cP)$, it is sufficient to note that there does not exist an inner interval of $\cP$ having $(a,b)$ and another vertex $w\in F\setminus \{(i,j)\}$ as anti-diagonal corners; in fact, if such an inner interval of $\cP$ exists, then $w\in \mathcal{N}$, which is a contradiction. To prove the maximality of $H$, we observe that if $H\subset K$ for some face $K$ of $\Delta(\cP)$ then there exists a facet $K_{\mathrm{max}}$ of $\Delta(\cP)$ such that $H\subset K_{\mathrm{max}}$ , so $\vert K_{\mathrm{max}}\vert >\vert H\vert =\vert F\vert$, which is a contradiction with the pureness of $\Delta(\cP)$. Therefore $H$ is a facet of $\Delta(\cP)$. Observe that $F <_{\mathrm{lex}} H$ since we replace $(i,j)$ in $F$ with $(a,b)$, so $H\in \cS(F)$. Hence $F\backslash\{(i,j)\} \in \langle \cS(F)\rangle$ and $\langle \cK_F\rangle\subseteq \langle \cS(F)\rangle \cap \langle F\rangle$.
 	
 	\begin{figure}[h!]
 		\centering
 		\includegraphics[scale=1]{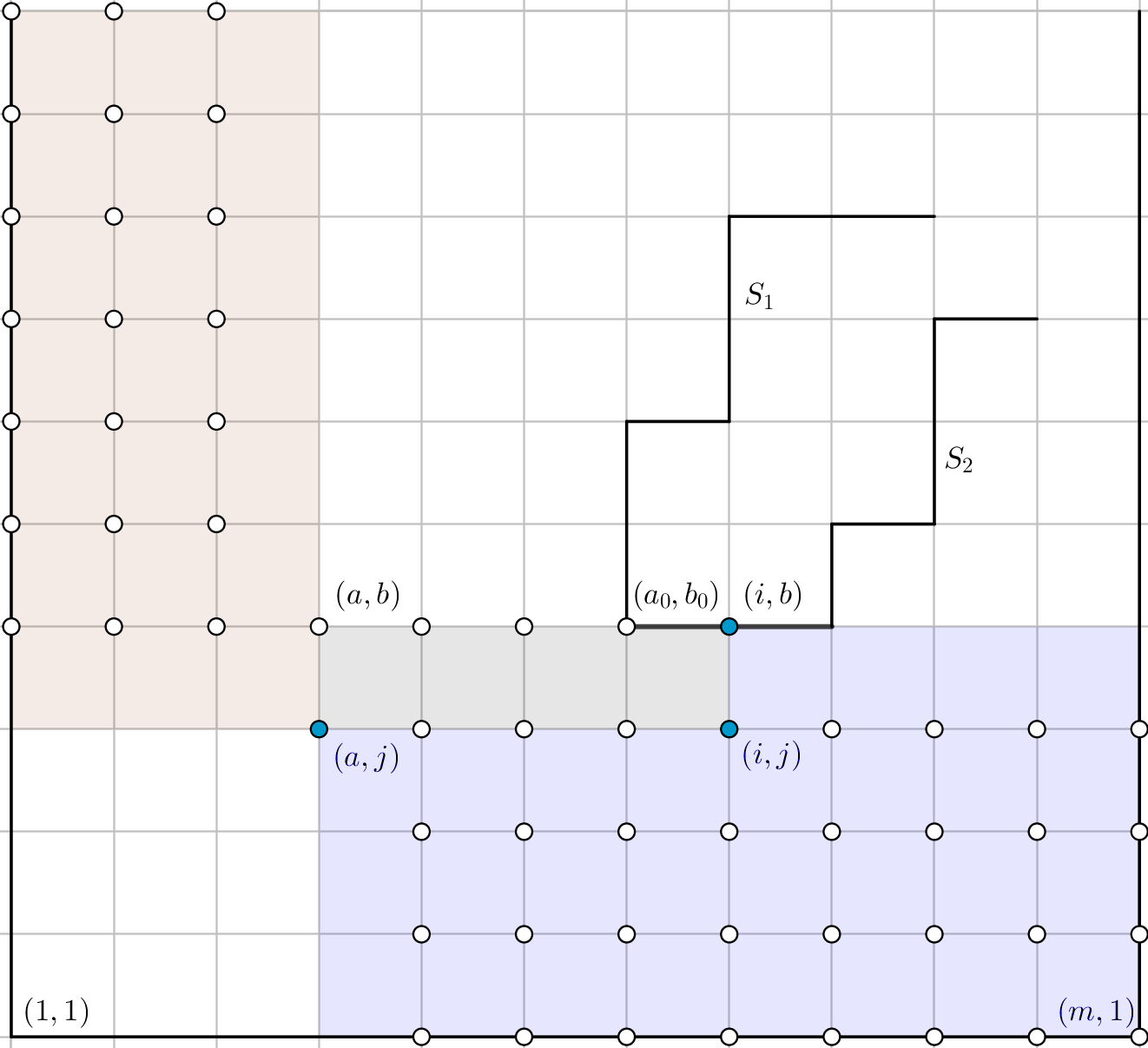}
 		\caption{An arrangement of intervals.}
 		\label{Figure: proof for shelling order}
 	\end{figure}

   \noindent Now, we prove $\langle \cS(F)\rangle \cap \langle F\rangle\subseteq \langle \cK_F\rangle$. Let $G$ be in $\langle \cS(F)\rangle \cap \langle F\rangle$. Since $G\in \langle F\rangle$, then $G\subseteq F$. Moreover, $G\neq F$ because $G\in \cS(F)$ and $F\notin \cS(F)$, so $G\subset F.$ Therefore, $G=F\backslash \{v_h:h\in [t]\}$, where $t\in [\vert F\vert]$ and $v_h\in F$ for all $h\in [t]$.  We discuss two cases.
   \begin{itemize}
       \item[(a)] Assume that $t=1$, so $G=F\setminus\{v_1\}$. We set $v_1=(i,j)$. Since $G=F\backslash \{v_1\}\in \langle \cS(F)\rangle$, then there exists a facet $H\in \cS(F)$ such that $F\backslash \{v_1\}\subset H$. Moreover, we recall that $\Delta(\cP)$ is pure and $F$ and $H$ are two facets of $\Delta(\cP)$ with $F<_{\mathrm{lex}}H$, so we can obtain $H$ from $G$ adding a vertex $w=(k,l)$, where $l>j$ or $j=l$ and $k>i$. We want to show that $v_1$ is the lower right corner of a step in $F$. Suppose by contradiction that $v_1$ is not the lower right corner of a step in $F$. With reference to Discussion \ref{Discussion: facet}, if we add to $G=F\setminus\{v_1\}$ a vertex $w=(k,l)$ (with $l>j$ or $j=l$ and $k>i$) then we find an inner interval of $\cP$ having $w$ and a vertex in $F\setminus\{v_1\}$ as anti-diagonal corners, so we have a contradiction with (2) of Proposition \ref{Prop: Grobner basis of frame polyomino and C.M.}. The operation to get $H$ from $G$ adding a vertex $w=(k,l)$ (with $l>j$ or $j=l$ and $k>i$) can be done just when $v_1$ is the lower right corner of a step $F'$ of $F$. In fact, in such a case, we can replace $v_1$ in $F$ with the anti-diagonal corner of the inner interval given by the step $F'$ to get $H$. Hence $v_1$ is necessarily the lower right corner of a step of $F$.
    \item[(b)] Assume that $t>1$, so $G=F\setminus\{v_1,v_2,\dots,v_t\}.$ Arguing as in case (a), it is easy to show that there exists $q\in [t]$ such that $v_{q}$ is the lower right corner of a step of $F$.
    \end{itemize}
    Hence we have that $G\in \langle \cK_F\rangle$ and $\langle \cS(F)\rangle \cap \langle F\rangle\subseteq \langle \cK_F\rangle$. In conclusion $\langle \cS(F)\rangle \cap \langle F\rangle = \langle \cK_F\rangle$.\\
    (2) It follows easily from (1) and Proposition \ref{Thm:McMullen-Walkup}. 
	\end{proof}

    \begin{rmk}\rm 
    The previous theorem does not hold in general. Consider the polyomino in Figure \ref{Figure: grid} and assume that $V(\cP)=[(1,1),(6,4)]$. Observe that $\cP$ is a prime polyomino (see \cite{Trento}) and $G(\cP)$ forms the reduced Gr\"obner basis of $I_{\cP}$ with respect to $<$, so $\Delta(\cP)$ is a shellable simplicial complex. The first three facets of $\Delta(\cP)$ lexicographically ordered in descending are:
    \begin{itemize}
        \item $F_0=\{(6, 4),(5, 4),(4, 4),(3, 4),(2, 4),(1, 4),(1, 3),(5, 2),(3, 2),(1, 2),(1, 1)\}$;
        \item $F_1=\{(6, 4),(5, 4),(4, 4),(3, 4),(2, 4),(1, 4),(1, 3),(5, 2),(3, 2),(3, 1),(1, 1)\}$;
        \item $F_2=\{(6, 4),(5, 4),(4, 4),(3, 4),(2, 4),(1, 4),(1, 3),(5, 2),(5, 1),(3, 1),(1, 1)\}$.  
    \end{itemize}
    Note that $F_2$ contains the two steps $\{(1,1),(3,1),(3,4)\}$ and $\{(3,1),(5,1),(5,2)\}$ but $\langle F_0,F_1\rangle\cap \langle F_2\rangle$ is generated just by $\{(6, 4), (5, 4), (4, 4), (3, 4), (2, 4), (1, 4), (1, 3), (5, 2), (3, 1), (1, 1)\}=F_2\setminus\{(5,1)\}$ and not by $F_2\setminus\{(3,1)\}$ and $F_2\setminus\{(5,1)\}$.

     \begin{figure}[h!]
    	\centering
    	\includegraphics[scale=1.1]{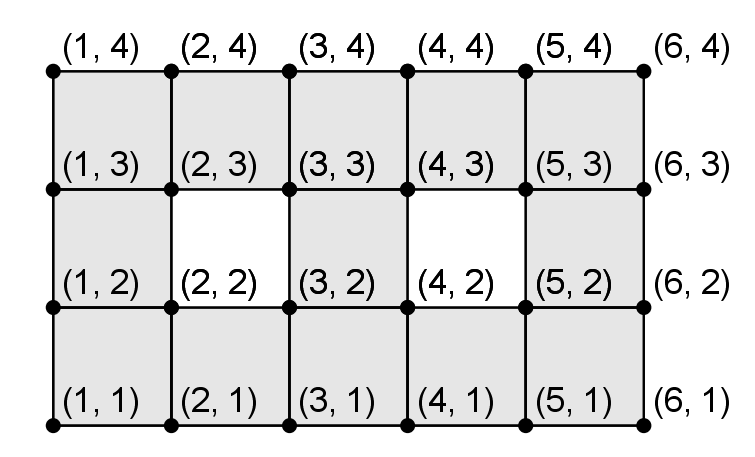}
    	\caption{A polyomino.}
    	\label{Figure: grid}
    \end{figure}
    \end{rmk}

    \section{Hilbert series and rook polynomial of frame polyominoes}

    \noindent In this section we study the Hilbert-Poincar\'e series of the coordinate ring attached to a frame polyomino. Let us start by introducing some definitions, following \cite{Parallelogram Hilbert series}.\\
    Let $\cP$ be a polyomino. Two rooks $R_1$ and $R_2$ are in \textit{attacking position} in $\cP$ if there exists a block $[A,B]$ of $\cP$ such that $R_1$ and $R_2$ are placed in $A$ and $B$ respectively. In such a case we say that $R_1$ and $R_2$ are two \textit{attacking rooks}. Moreover, two rooks are in \textit{non-attacking position} in $\cP$ (or they are two \textit{non-attacking rooks}) if they are not in attacking position in $\cP$. For instance see Figure \ref{Figure: exa attacking rooks}.
    
    \begin{figure}[h]
		\centering
		\subfloat[Attacking rooks]{\includegraphics[scale=0.8]{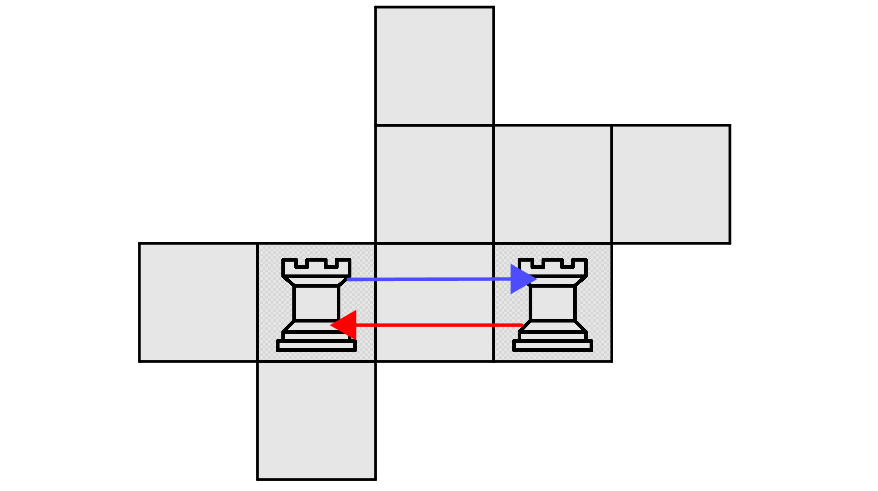}}\quad
         \subfloat[Non-attacking rooks]{\includegraphics[scale=0.8]{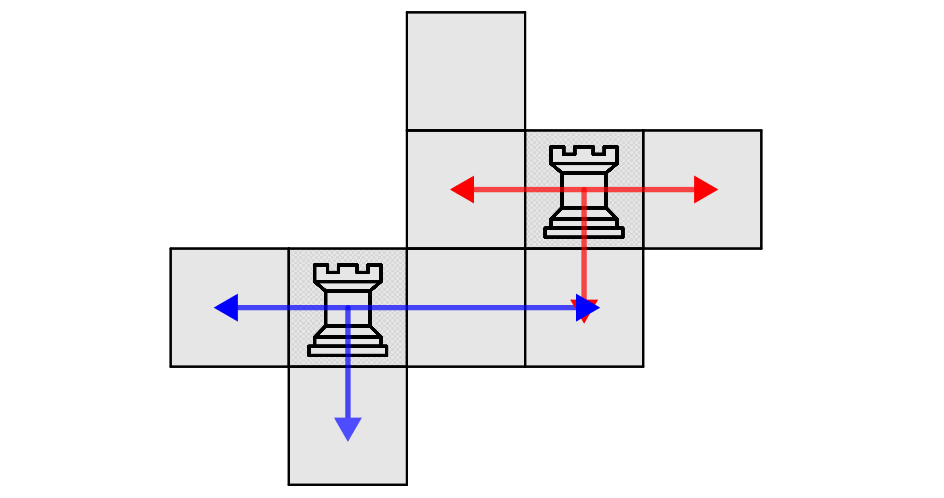}}	
  \caption{Positions of two rooks in a polyomino.}
		\label{Figure: exa attacking rooks}
	\end{figure}

    \noindent A \textit{$j$-rook configuration} in $\cP$ is a configuration of $j$ rooks which are arranged in $\cP$ in non-attacking positions. Figure~\ref{Figura:esempio rook configuration} shows a 6-rook configuration.
    
	\begin{figure}[h]
		\centering
		\includegraphics[scale=0.85]{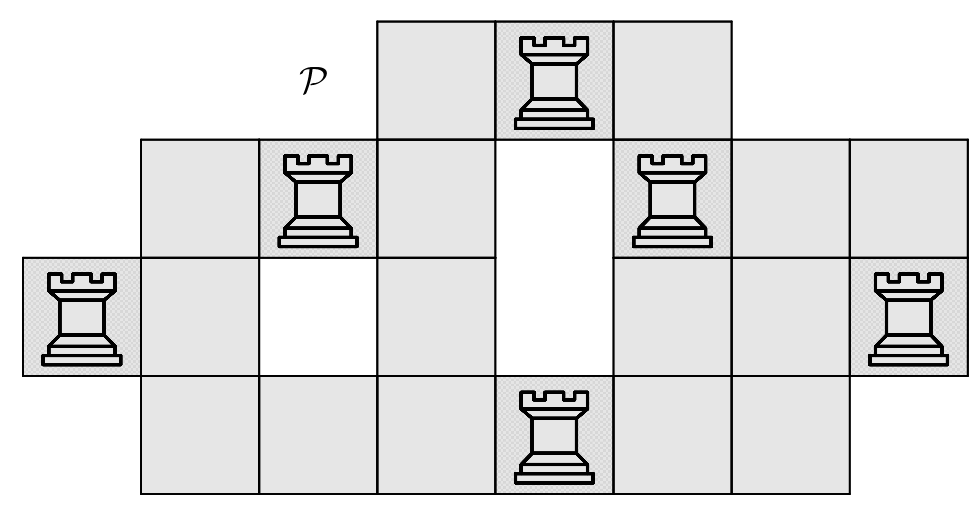}
		\caption{An example of a $6$-rook configuration in $\cP$.}
		\label{Figura:esempio rook configuration}
	\end{figure}
	
	\noindent The rook number $r(\cP)$ is the maximum number of rooks that can be placed in $\cP$ in non-attacking position. We denote by $\cR_j$ the set of all $j$-rook configurations in $\cP$ for all $j\in [r(\cP)]$ and we set conventionally $R_0=\emptyset$. Note that $\cR_0\cup \cR_1\dots \cup\cR_{r(\cP)}$ is a simplicial complex, called \textit{rook complex}.\\
 Two non-attacking rooks in $\cP$ are said to be in \textit{switching position} or they are called \textit{switching rooks} if they are placed in the diagonal (resp. anti-diagonal) cells of $\cP_I$, where $I$ is an inner interval of $\cP$.
 In such a case we say that the rooks are in a diagonal (resp. anti-diagonal) position.\\
 Fix $j\in \{0,\dots, r(\cP)\}$. Let $F\in \cR_j$ and $R_1$ and $R_2$ be two switching rooks of $F$ in diagonal (resp. anti-diagonal) position in $\cP_I$, where $I$ is an inner interval of $\cP$. Let $R_1'$ and $R_2'$ be the rooks in anti-diagonal (resp. diagonal) cells of $\cP_I$. Then the set $(F\backslash \{R_1, R_2\}) \cup \{R_1', R_2'\}$ belongs to $\cR_j$. The operation of replacing $R_1$ and $R_2$ by $R_1'$ and $R_2'$ is called \textit{switch of $R_1$ and $R_2$}. This induce the following equivalence relation $\sim$ on $\cR_j$: let $F_1, F_2 \in \cR_j$, so $F_1\sim F_2$ if $F_2$ can be obtained from $F_1$ after some switches. Look at Figure \ref{Figure: exa switch attacking rooks} for an example of four $3$-rook configurations which are equivalent with respect to $\sim$. 
 
 \begin{figure}[h!]
		\centering
		\subfloat{\includegraphics[scale=0.8]{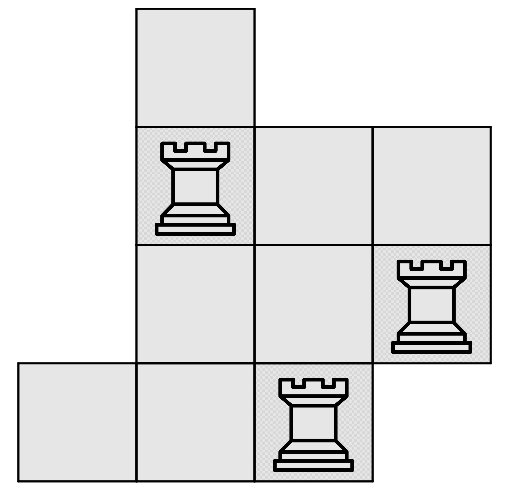}}\qquad
         \subfloat{\includegraphics[scale=0.8]{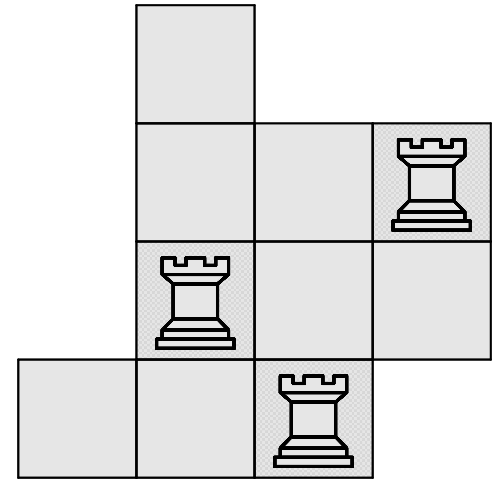}}\qquad
        \subfloat{\includegraphics[scale=0.8]{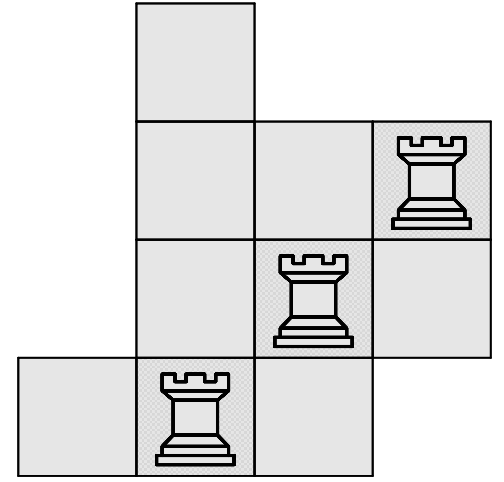}}	\qquad
        \subfloat{\includegraphics[scale=0.8]{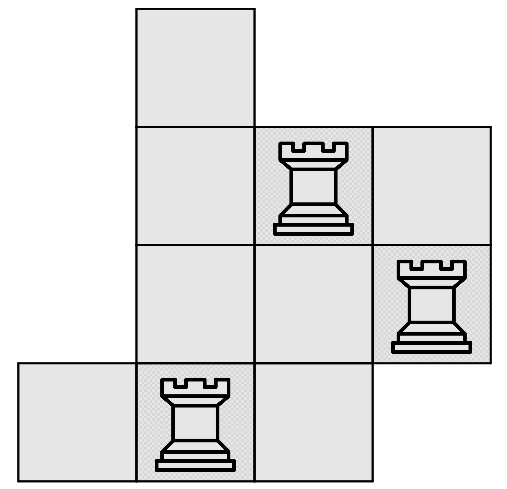}}	
  \caption{}
		\label{Figure: exa switch attacking rooks}
	\end{figure}
 
 \noindent Let $\tilde{\cR}_j = \cR_j/\sim$ be the quotient set. We set $\tilde{r}_j=\vert \tilde{\cR}_j\vert $ for all $j\in [r(\cP)]$; conventionally $\tilde{r}_0=1$. The \textit{switching rook-polynomial} of $\cP$ is the polynomial in $\mathbb{Z}[t]$ defined as $\tilde{r}_{\cP}(t)=\sum_{j=0}^{r(\cP)}\tilde{r}_jt^j$.	\\
 
\noindent In the following proposition, we show a natural representative of an equivalence class of $\tilde {\cR}_j$ associated with a frame polyomino.
 
	\begin{prop} \label{Prop: equivalent class for switching rooks}
 
    Let $\cP$ be a frame polyomino and $\cR$ be a $j$-rook configuration in $\cP$, with $j\in [r(\cP)]$. Then there exists a $j$-rook configuration $\cC$ in $\cP$ such that $\cC \sim \cR$ and any two switching rooks of $\cC$ are placed in diagonal position. 	
 \end{prop}
	
	\begin{proof}
  If $j=1$, then the claim follows trivially. Suppose that $2\leq j\leq r(\cP)$. We distinguish the following cases with reference to Figure \ref{Figure: proof for equivalance rook}.
    \begin{itemize}
    \item[Case 1] Assume that no rook of $\cR$ is in $\cQ$. Then, every rook of $\cR$ is placed in $\cP_1\setminus \cQ$ or in $\cP_2 \setminus \cQ$. Moreover, the rooks in $\cP_1\setminus \cQ$ are never in attacking and switching position with the rooks in $\cP_2 \setminus \cQ$. Hence, the claim follows immediately from \cite[Lemma 3.12]{Parallelogram Hilbert series} applying to $\cP_1\setminus \cQ$ and to $\cP_2 \setminus \cQ$. 
    \item[Case 2] Suppose that the rooks of $\cR$ belong just to $\cQ$. Then the claim follows trivially from \cite[Lemma 3.12]{Parallelogram Hilbert series}, since the cell intervals $\cP_{[(1,1),(a_0,b_0)]}$ and $\cP_{[(a_k,b_k),(m,n)]}$ associated respectively to $[(1,1),(a_0,b_0)]$ and $[(a_k,b_k),(m,n)]$ are parallelogram polyominoes.
    \item[Case 3] The same conclusion holds if all rooks of $\cR$ are either in $\cP_1$ or in $\cP_2$.
    \item[Case 4] Suppose that there exists a rook of $\cR$ in $\cP_1\setminus \cQ$, another one 
    in $\cP_2\setminus \cQ$ and another in $\cQ$. We consider non-empty sets of rooks $\cR_1,\cR_2,\tilde{\cR} \subset \cR$ such that the rooks in $\cR_1$, $\cR_2$ and $\tilde{\cR}$ are placed in $\cP_1$, $\cP_2$ and $\cQ$ respectively, $\cR=\cR_1\cup\cR_2\cup \tilde{\cR}$ and $\tilde{\cR}=\cR_1\cap\cR_2$. We proceed in two steps.
    \begin{itemize}
        \item[Step 1.] First of all, we focus on $\cP_1$ and $\cR_1$. Since $\cP_1$ is a parallelogram polyomino, then from \cite[Lemma 3.12]{Parallelogram Hilbert series} there exists a $\vert \cR_1\vert$-rook configuration $\cC_1$ in $\cP_1$ such that $\cC_1\sim\cR_1$ and any two switching rooks in $\cC_1$ are placed in diagonal position. Denote by $\cW$ the set of the rooks of $\cC_1$ placed in $\cQ$. We show that $\vert \cW\vert =\vert \tilde{\cR}\vert$. With reference to Figure \ref{Figura:arragement switching rook}, let $R_1$ and $R_2$ be two switching rooks of $\cR_1$ in anti-diagonal position and $R_1'$ and $R_2'$ the related switching rooks in a diagonal position.   
    
    \begin{figure}[h]
		\centering
		\includegraphics[scale=0.7]{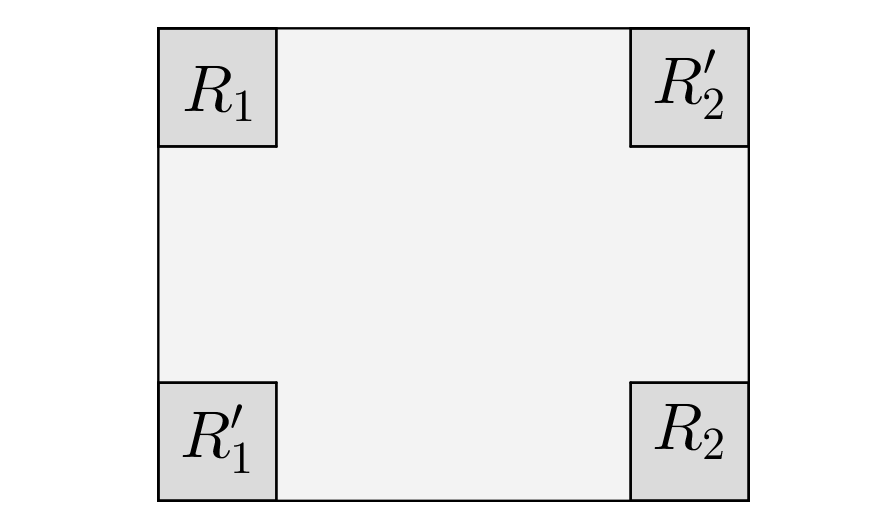}
		\caption{Placement of switching rooks of $\cR_1$.}
		\label{Figura:arragement switching rook}
	\end{figure}
    
    \noindent From the structure of $\cP_1$ the following remarks follow:
    \begin{itemize}
        \item[-] if $R_1$ and $R_2$ are placed in $\cP_1\setminus \cQ$, then $R_1'$ and $R_2'$ are also in $\cP_1\setminus\cQ$;
       \item[-] if $R_1$ and $R_2$ are in $\cQ$, then $R_1'$ and $R_2'$ are also in $\cQ$;
        \item[-] if $R_1$ and $R_2$ are placed respectively in $\cP_1\setminus \cP_{[(1,1),(a_0,b_0)]}$ and $\cP_{[(1,1),(a_0,b_0)]}$, then $R_1'$ and $R_2'$ are respectively in $\cP_{[(1,1),(a_0,b_0)]}$ and $\cP_1\setminus \cP_{[(1,1),(a_0,b_0)]}$;
        \item[-] if $R_1$ and $R_2$ are respectively in $\cP_1\setminus\cP_{[(a_k,b_k),(m,n)]}$ and $\cP_{[(a_k,b_k),(m,n)]}$, then $R_1'$ and $R_2'$ are respectively in $\cP_1\setminus\cP_{[(a_k,b_k),(m,n)]}$ and $\cP_{[(a_k,b_k),(m,n)]}$.
    \end{itemize}
      Therefore it easy to see that $\vert \cW\vert =\vert \tilde{\cR}\vert$. 
      \item[Step 2] Now, we consider $\cP_2$ and the $\vert \cR_2\vert$-rook configuration $(\cR_2\setminus \tilde{\cR})\cup \cW$ in $\cP_2$. As before, since $\cP_2$ is a parallelogram polyomino, there exists a $\vert\cR_2\vert$-rook configuration $\cC_2$ in $\cP_2$ such that $\cC_2\sim (\cR_2\setminus \tilde{\cR})\cup\cW$ and any two switching rooks in $\cC_2$ are placed in diagonal position. Moreover, if $\mathcal{Z}$ is the subset of $\cC_2$ of the rooks placed in $\cQ$, then $\vert \mathcal{Z}\vert=\vert \cW\vert =\vert \tilde{\cR}\vert$.
          \end{itemize}

    \noindent Set $\cC=(\cC_1\setminus \cW)\cup \cC_2$ and we show that $\cC$ is the desired $j$-rook configuration. Observe trivially that $\cC\sim \cR$ because we get $\cC$ from $\cR$ with suitable switches. We need to prove that any two switching rooks in $\cC$ are in a diagonal position. Let $T_1$ and $T_2$ be two switching rooks in $\cC$. We analyze the following cases.\\
    (1) If $T_1$ and $T_2$ are placed in $\cP_1\setminus \cQ$ then they are in diagonal position by Step 1, as well as if $T_1$ and $T_2$ are in $\cP_2$ by Step 2.\\
    (2) Suppose that $T_1$ is in $\cP_{[(1,1),(a_0,b_0)]}$ and $T_2$ in  $\cP_1\setminus \cQ$. We denote by $[B_1,B_r]$ the maximal vertical block of $\cP_{[(1,1),(a_0,b_0)]}$ where the rook $T_1$ is placed. After Step 1, a rook $R$ is placed in $[B_1,B_r]$. In fact, if no rook is in $[B_1,B_r]$ after Step 1, then there is no rook in $[B_1,B_r]$ after Step 2; but this is not possible since we are assuming that the rook $T_1$ is in $[B_1,B_r]$ after Step 2.
    Moreover, by Step 1 we have that $R$ and $T_2$ are in diagonal position (see Figure \ref{Figure: Case 4 proof} (A)). Now, observe that applying Step 2 we remove $R$ and we put $T_1$ in $[B_1,B_r]$. Since $R$ and $T_2$ were in diagonal position, then also $T_1$ and $T_2$ are in diagonal position (see Figure \ref{Figure: Case 4 proof} (B)). 
    \begin{figure}[h!]
		\centering
		\subfloat[After Step 1]{\includegraphics[scale=0.78]{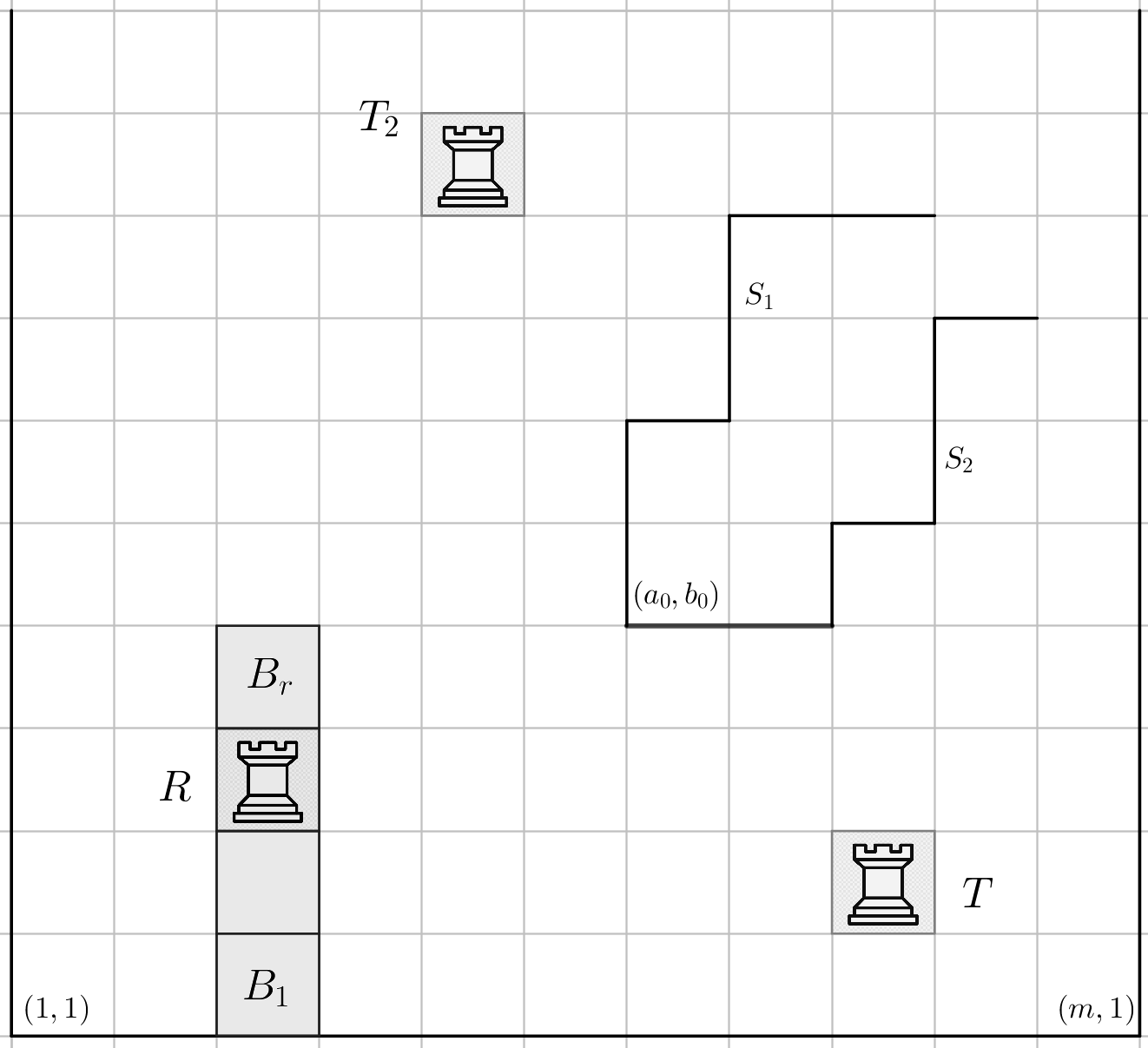}}\
         \subfloat[After Step 2]{\includegraphics[scale=0.78]{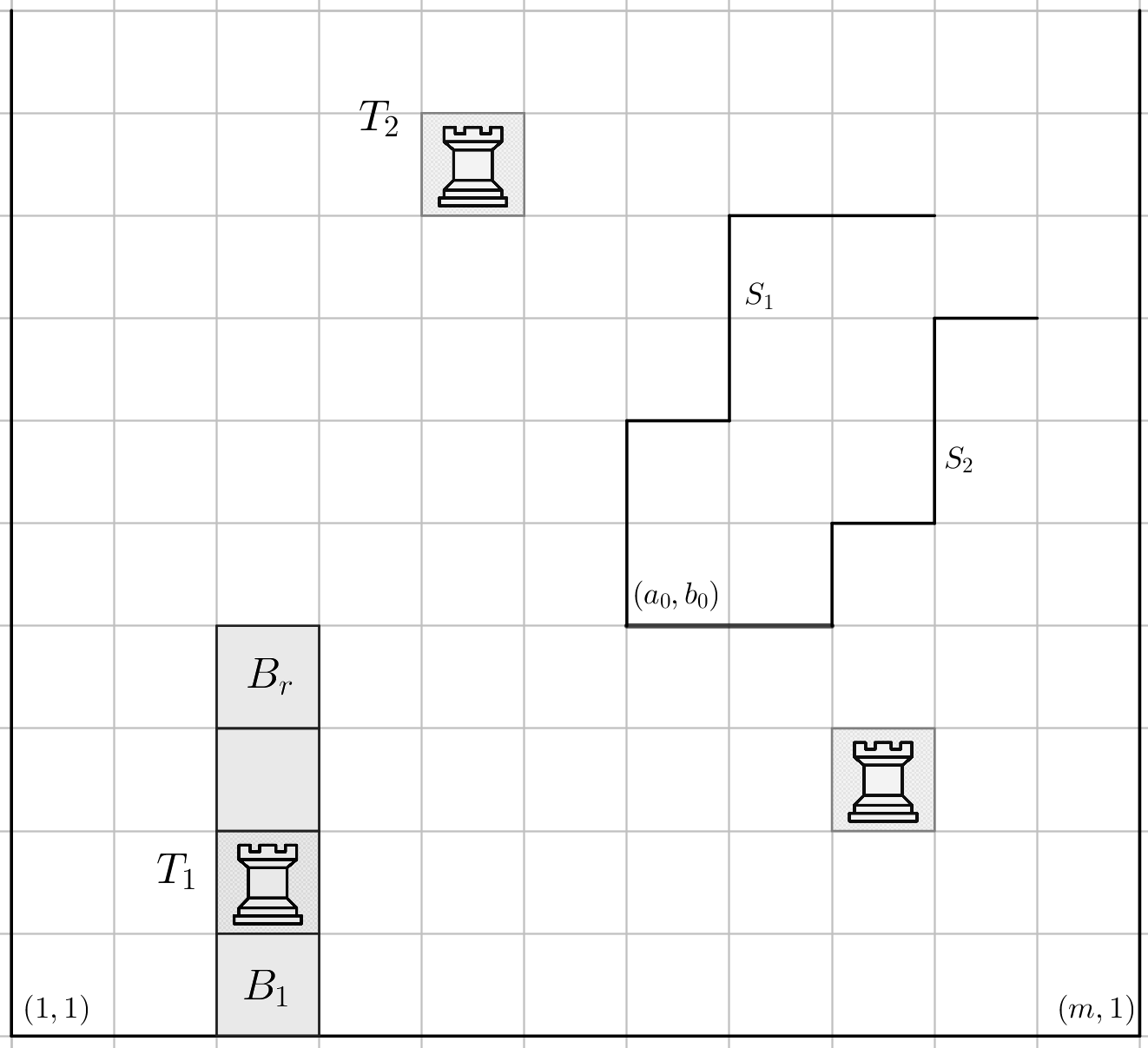}}	
  \caption{$T_1$ in $\cP_{[(1,1),(a_0,b_0)]}$ and $T_2$ in $\cP_1\setminus \cQ$.}
		\label{Figure: Case 4 proof}
	\end{figure}  
 
     \noindent (3) If $T_1$ is in $\cP_1\setminus \cQ$ and $T_2$ in $\cP_{[(a_k,b_k),(m,n)]}$, then we get the claim arguing as done in (2).\\
     
    \noindent Therefore, $\cC$ is a $j$-rook configuration such that $\cC \sim \cR$ and any two switching rooks of $\cC$ are placed in a diagonal position, as claimed.  	
    \end{itemize}

    \end{proof}
	
    \begin{defn}\rm 
    Let $\cP$ be a frame polyomino and $\cR$ be a $j$-rook configuration in $\cP$, with $j\in[r(\cP)]$. We say that the $j$-rook configuration $\cC$ defined in Proposition \ref*{Prop: equivalent class for switching rooks} is a \textit{canonical configuration} of $\cR$.
    \end{defn} 
   
    \noindent Now we provide some lemmas and a remark which is useful to prove Theorem \ref{Theorem: bijection between facets and canonical positions}.
    
	\begin{lemma}\label{Lemma: If two l.r.corners are on the Horiz. edge then...}
	Let $\cP$ be a frame polyomino and $\Delta(\cP)$ be the simplicial complex attached to $\cP$. Let $F$ be a facet of $\Delta(\cP)$ with $j$ steps, where $j\geq 2$.  
	If $v$ and $w$ are the lower right corners of two distinct steps of $F$ belonging on the same maximal horizontal edge interval of $\cP$ with $v<w$ then $v\in V(\cP_{[(1,1),(a_0,b_0)]})$ and $w\notin V(\cP_{[(1,1),(a_0,b_0)]})$.
	\end{lemma}
								
	\begin{proof}
	Let $G=\{g_1,v,g_2\}$ and $U=\{u_1,w,u_2\}$ be the two steps of $F$ having respectively $v$ and $w$ as lower right corners. 
    Since $v$ and $w$ are on the same maximal horizontal edge interval $H$ of $\cP$ and $v<w$, then $g_1,u_1\in H$ and $g_1<v\leq u_1<w$. By contradiction suppose that $v\notin V(\cP_{[(1,1),(a_0,b_0)]})$ or $w\in V(\cP_{[(1,1),(a_0,b_0)]})$. We examine the three cases. If $v\notin V(\cP_{[(1,1),(a_0,b_0)]})$ and $w\in V(\cP_{[(1,1),(a_0,b_0)]})$, then necessarily $w<v$, which is a contradiction with the assumption that $v<w$. Assume that $v\in V(\cP_{[(1,1),(a_0,b_0)]})$ and $w\in V(\cP_{[(1,1),(a_0,b_0)]})$. In such a case, for the structure of $\cP$, $g_2$ and $w$ are the anti-diagonal corner of an inner interval of $\cP$ but $g_2, w\in F$, so we get a contradiction. The same comes when 
	$v\notin V(\cP_{[(1,1),(a_0,b_0)]})$ and $w\notin V(\cP_{[(1,1),(a_0,b_0)]})$. 
	In every case, a contradiction arises. Hence the conclusion follows
 	\end{proof}
 
 	\begin{lemma}\label{Lemma: If two l.r.corners are on the Vert. edge then...}
 	Let $\cP$ be a frame polyomino and $\Delta(\cP)$ be the simplicial complex attached to $\cP$. Let $F$ be a facet of $\Delta(\cP)$ with $j$ steps, with $j\geq 2$. 
 	If $v$ and $w$ are the lower right corners of two distinct steps of $F$ belonging on the same maximal vertical edge interval of $\cP$ with $w<v$ then $v\in V(\cP_{[(a_k,b_k),(m,n)]})$ and $w\notin V(\cP_{[(a_k,b_k),(m,n)]})$.
 	\end{lemma}
	
	\begin{proof}
	The claim follows arguing similarly as done in the proof of Lemma \ref{Lemma: If two l.r.corners are on the Horiz. edge then...}.
	\end{proof}

	\begin{rmk}\rm \label{Remark: facets vs chains in parallelogram}
	Let $\cP$ be a parallelogram polyomino and $\Delta(\cP)$ be the simplicial complex attached to $\cP$. Recall that, for all $j\geq 0$, every maximal chain of $\cP$ with $j$ descents as distributive lattice is a facet of $\Delta(\cP)$ with $j$ steps, and vice versa. As a consequence of \cite[Proposition 3.11, Lemma 3.12]{Parallelogram Hilbert series}, we observe that there is a one-to-one correspondence between the canonical configurations in $\cP$ of $j$ rooks and the facets of $\Delta(\cP)$ with $j$ steps.
	\end{rmk}

 \noindent  Now we are ready to prove one of the crucial results of this work, which provides a bijection between the facets with $j$ steps of the simplicial complex attached to a frame polyomino $\cP$ and the canonical configurations of $j$ rooks in $\cP$.  
	
	\begin{thm}\label{Theorem: bijection between facets and canonical positions}
	Let $\cP$ be a frame polyomino and $\Delta(\cP)$ be the simplicial complex attached to $\cP$. For all $j\geq 0$, there exists a bijection between
	the facets with $j$ steps of $\Delta(\cP)$ and the canonical configurations in $\cP$ of $j$ rooks.  
	\end{thm}
	
	\begin{proof}
	The first part of the proof is devoted to defining a desired bijective function that uniquely assigns to a facet of $\Delta(\cP)$ a canonical configuration in $\cP$. Let $F$ be a facet of $\Delta(\cP)$ with $j$ steps. 
    \begin{itemize}
        \item  If $j=0$, then $F$ has no step. Observe from Discussion \ref{Discussion: facet} that $F=F_I\cup S^*$, where $F_I\cup S^*$ is the facet defined in the proof of Proposition \ref{Prop: Grobner basis of frame polyomino and C.M.}. We associate to $F$ the $0$-rook configuration, which is the empty set. 
        \item If $j=1$, then the facet $F$ has just one step, whose lower right corner is denoted by $v_F$. In such a case, we attach to $F$ the $1$-rook configuration defined by placing a rook in the cell of $\cP$ having $v_F$ as the lower right corner. 
        \item Assume that $j \geq 2$. Consider a step $F'$ of $F$ having $w$ as the lower right corner and we denote by $H_w$ and $V_w$ respectively the maximal horizontal and vertical edge intervals of $\cP$ containing $w$. We look at the following two possibilities.
 \begin{enumerate}
     \item There does not exist any step with a lower right corner belonging to $H_w$ and to $V_w$. In such a case, we place a rook in the cell of $\cP$ having $w$ as the lower right corner. 
    \item There exists a step with a lower right corner $v=(i,h)$ belonging to $H_w$ or to $V_w$. 
    \begin{itemize}
        \item[a)] Assume that $v\in H_w$. It is not restrictive to suppose that $v<w$ because the other case will follow similarly. From Lemma \ref{Lemma: If two l.r.corners are on the Horiz. edge then...}, we have that $v\in V(\cP_{[(1,1),(a_0,b_0)]})$ and $w\notin V(\cP_{[(1,1),(a_0,b_0)]})$. Therefore we assign a rook to the cell of $\cP$ having $w$ as the lower right corner and another one to the cell having $(i,b_0)$ as a lower right corner.
        \item[b)] Assume that $v\in V_w$. We may take $w<v$ because the other case is similar. From Lemma \ref{Lemma: If two l.r.corners are on the Vert. edge then...}, we get $v\in V(\cP_{[(a_k,b_k),(m,n)]})$ and $w\notin V(\cP_{[(a_k,b_k),(m,n)]})$. Therefore we attach a rook to the cell of $\cP$ having $w$ as the lower right corner and another one to the cell having $(a_k,h)$ as a lower right corner. 
    \end{itemize}
 \end{enumerate}
    
    \end{itemize}
In this way, we define a configuration $\cR$ of $j$ rooks in $\cP$, related to the facet $F$ of $\Delta(\cP)$ with $j$ steps. We need to show that $\cR$ is a canonical configuration in $\cP$. Firstly, we prove that every pair of rooks of $\cR$ is in non-attacking position. Let $R_1, R_2\in \cR$ and $F_1=\{a,v_1,b\}, F_2=\{c,v_2,d\}$ be the two steps related respectively to $R_1$ and $R_2$ having $v_1=(i_1,j_1)$ and $v_2=(i_2,j_2)$ as lower right corners, with $v_1<v_2$. The only case that we need to examine is when $v_1$ and $v_2$ are on the same maximal edge interval of $\cP$. We may assume that $j_1=j_2$ because the case $i_1=i_2$ can be shown similarly. From Lemma \ref{Lemma: If two l.r.corners are on the Horiz. edge then...}, we have that $v_1\in V(\cP_{[(1,1),(a_0,b_0)]})$ and $v_2\notin V(\cP_{[(1,1),(a_0,b_0)]})$. Hence, $R_1$ is in the cell of $\cP$ having $(i_1,b_0)$ as lower right corner and $R_2$ is in the cell of $\cP$ with $v_2$ as lower right corner, so $R_1$ and $R_2$ are not in attacking position. Moreover, we cannot find any rook in the vertical and horizontal blocks of $\cP$ containing $R_1$ or $R_2$. We show just that there is no rook in the vertical and horizontal blocks of $\cP$ containing $R_1$, because the other case for $R_2$ can be shown similarly. By contradiction, if there is a rook different from $R_1$ in the vertical block of $\cP$ containing $R_1$, then there exists a step $G=\{p,q,r\}$, where $q$ is its lower right corner with $q=(i_1,h_1)$ and $1\leq h_1\leq n$. But in such a case, we have that either $p,v_1$ or $a,q$ are the anti-diagonal corners of an inner interval of $\cP$, a contradiction. In a similar way, we show that there is no rook in the horizontal block of $\cP$ containing $R_1$.  Hence all pairs of rooks in $\cR$ are in non-attacking position. Moreover, we observe that two switching rooks of $\cR$ cannot be in anti-diagonal position; in fact, by contradiction, if two switching rooks of $\cR$ are in anti-diagonal position then the lower right corners of the related steps are the anti-diagonal corners of an inner interval of $\cP$, which is a contradiction with (2) of Proposition \ref{Prop: Grobner basis of frame polyomino and C.M.}. We conclude that $\cR$ is a canonical configuration.\\
Set $\cR=\cC_F$ and denote by $\cF_{\cP,j}$ the set of the facets of $\Delta(\cP)$ with $j$ steps and by $\mathfrak{C}_{\cP,j}$ the set of all canonical configurations in $\cP$ of $j$ rooks. We introduce the map $\psi: \cF_{\cP,j}\rightarrow\mathfrak{C}_{\cP,j}$ where $\psi(F)$ is the canonical configuration $\cC_F$ defined before, for all $F\in \cF_{\cP,j}$.
	We prove that $\psi$ is bijective.\\
 Firstly, we show that $\psi$ is injective. Let $F_1, F_2\in \cF_{\cP,j}$ such that $F_1\neq F_2$.  We prove that $\cC_{F_1}\neq \cC_{F_2}$. Since $F_1\neq F_2$, there exists $a\in V(\cP)$ such that $a\in F_1$ and $a\notin F_2$. Set $a=(a_x,a_y)$. Firstly, assume that $a\in V(\cP_{[(1,1),(a_0,b_0)]})\setminus [(1,b_0),(a_0,b_0)]$. We distinguish two cases. 

 \begin{itemize}
 
    \item[Case 1] Suppose that $a$ is the lower right corner of a step $\{p,a,q\}$ of $F_1$. Set $p=(p_x,a_y)$ and $q=(a_x,q_y)$. It follows from Discussion \ref{Discussion: facet} that $p_x=a_x-1$ and $q_y$ is either $a_y+1$ or $b_0+1$. We examine the following two sub-cases.
    
    \begin{itemize}
    \item[(1)] If $q_y=a_y+1$, then a rook $R_1$ of $\cC_{F_1}$ is in the cell $C$ of $\cP$ with $a$ as lower right corner. Observe from the definition of $\psi$ that the only possibility in order to $R_1\in \cC_{F_2}$ is that  
    $\{p,a,q\}$ is a step of $F_2$. Since $a\notin F_2$, $\{p,a,q\}$ is not a step of $F_2$, so $R_1\notin \cC_{F_2}$, that is $\cC_{F_1}\neq \cC_{F_2}$.
    \item[(2)] If $q_y=b_0+1$, then from Discussion \ref{Discussion: facet} there exists a step of $F_1$ with lower right corner $a'$ such that $a$ and $a'$ are on the same maximal horizontal edge interval of $\cP$ and $a'\notin V(\cP_{[(1,1),(a_0,b_0)]})$. Hence a rook $R_1$ of $\cC_{F_1}$ is in the cell of $\cP$ having $(a_x,b_0)$ as lower right corner and another rook $R_2$ of $\cC_{F_1}$ is in the cell of $\cP$ having $a'$ as lower right corner. We show that $R_1$ or $R_2$ do not belong to $\cC_{F_2}$. If $R_1\notin \cC_{F_2}$, then we have finished. Suppose that $R_1\in \cC_{F_2}$ and we prove that $R_2\notin\cC_{F_2}$. Since $a\notin F_2$ and $R_1\in \cC_{F_2}$, then from Discussion \ref{Discussion: facet} the only possibility is that $\{(a_x-1,b_0),(a_x,b_0),(a_x,b_0+1)\}$ is a step of $F_2$. As a consequence, the vertices in $\{(i,j)\in V(\cP):i\geq a_x,j<b_0\}$ do not belong to $F_2$, so no rook of $\cC_{F_2}$ is placed in a cell of the cell interval $\cP_{[(a_0,1),(m,b_0)]}$ and $R_2\notin \cC_{F_2}$. Therefore, $\cC_{F_1}\neq \cC_{F_2}$. 
    \end{itemize}
    
    \item[Case 2] Suppose that $a$ is not the lower right corner of a step of $F_1$. Recall that $a\neq (1,1)$ since $a\notin F_2$. We examine the following sub-cases.
    
    \begin{itemize}
        \item[(1)] If $1<a_x\leq a_0$ and $a_y=1$, then $[(1,1),(a_x,1)]\subset F_1$ as explained in Discussion \ref{Discussion: facet}. We show firstly that $(a_x+1,1)$ belongs to $F_1$. Suppose by contradiction that $(a_x+1,1)\notin F_1$. From Discussion \ref{Discussion: facet} it follows that $(a_x,2)\in F_1$ or $(a_x,b_0+1)\in F_1$ and no vertex in $[(a_x,2),(a_x,b_0)]$ is in $F_1$. In both cases $(a_x,1)$ is a lower right corner of a step of $F_1$, which is a contradiction. Hence $(a_x+1,1)\in F_1$. Let $k_1$ be the maximum integer in $\{a_x+1,\dots,a_0\}$ such that $(k_1,1)\in F_1$. So, we have only two possibilities: either $(k_1,2)\in F_1$ or $(k_1,2)\notin F_1$. If $(k_1,2)\in F_1$ then $\{(k_1-1,1),(k_1,1),(k_1,2)\}$ is a step of $F_1$ so a rook $Q_1$ of $\cC_{F_1}$ is in the cell of $\cP$ having $(k_1,1)$ as lower right corner. If $(k_1,2)\notin F_1$, then it easily follows from Discussion \ref{Discussion: facet} that $\{(k_1-1,1),(k_1,1),(k_1,b_0+1)\}$ is a step of $F_1$ and there exists another step of $F_1$ whose lower right corner $v$ belongs to $[(a_0+1,1),(m,1)]$. Hence, from the definition of $\psi$, there is a rook $T_1$ of $\cC_{F_1}$ in the cell of $\cP$ with $(k_1,b_0)$ as lower right corner and another rook $T_2$ of $\cC_{F_1}$ in the cell of $\cP$ with $v$ as lower right corner. \\
        Now, we consider $F_2$. Let $N=\{(i,j)\in V(\cP): i<a_x,j>1\}$. We prove that there exists a vertex $w$ in $N$ belonging to $F_2$. In fact, by contradiction, if there does not exist any vertex in $N$ belonging to $F_2$, then there is no inner interval of $\cP$ having as anti-diagonal corners $a$ and a vertex in $N$, hence $a$ belongs to $F_2$ due to the maximality of $F_2$. But this is a contradiction because $a\notin F_2$. Therefore, let $w=(w_x,w_y)\in N\cap F_2$. If $w\in [(1,2),(a_x-1,b_0)]$, then no vertex in $[(w_x+1,1),(m,w_y-1)]$ is in $F_2$ so no rook of $\cC_{F_2}$ is in a cell of $\cP_{[(w_x,1),(m,w_y)]}$. Hence neither $Q_1$ or $T_2$ belong to $\cC_{F_2}$, that is $\cC_{F_1}\neq \cC_{F_2}$. If $w\in N\setminus [(1,2),(a_x-1,b_0)]$, then any vertex in $[(w_x+1,1),(a_0,b_0)]$ is in $F_2$ so no rook of $\cC_{F_2}$ is in a cell of $\cP_{[(w_x+1,1),(a_0,b_0+1)}$ and neither $Q_1$ or $T_1$ belong to $\cC_{F_2}$, so $\cC_{F_1}\neq \cC_{F_2}.$
        \item[(2)] If $a_x=1$ and $1<a_y< b_0$, then we use similar arguments as done in the previous sub-case (1) to prove that no rook of $\cC_{F_1}$ is in a cell of $\cP_{[(1,1),(m,a_y)]}$ and at least one rook of $\cC_{F_2}$ is in a cell of a sub-polyomino of $\cP_{[(1,1),(m,a_y)]}$.
        \item[(3)] If $1<a_x\leq a_0$ and $1<a_y< b_0$, then we can argue as done in (1) for $F_1$. For the discussion about $F_2$, we may consider $N_1=\{(i,j)\in V(\cP): i<a_x,j>a_y\}$ and $N_2=\{(i,j)\in V(\cP): i>a_x,j<a_y\}$. In particular, if $N_1\cap F_2\neq \emptyset$, so we get the claim arguing similarly as done in the sub-case (1). If $N_1\cap F_2= \emptyset$ and $N_2\cap F_2\neq \emptyset$, then we can argue as in sub-case (2). In both we get the desired claim.  
    \end{itemize}
 \end{itemize}
 \noindent The two examined cases lead to $\cC_{F_1}\neq \cC_{F_2}.$ Moreover, all the other situations when $a\notin V(\cP_{[(1,1),(a_0,b_0)]})\setminus[(1,b_0),(a_0,b_0)]$ give us $\cC_{F_1}\neq \cC_{F_2}$, using an approach similar to the previous one. Hence we conclude that $\cC_{F_1}\neq \cC_{F_2}$, so $\psi$ is injective.\\
\noindent We prove that $\psi$ is surjective. Let $\cT$ be a canonical configuration of $j$ rooks and we prove that there exists a facet $F$ of $\Delta(\cP)$ of $j$ steps such that $\psi(F)=\cT$. If $j=0$, then we set $F=F_I\cup S^*$, where $F_I\cup S^*$ is the facet defined in the proof of Proposition \ref{Prop: Grobner basis of frame polyomino and C.M.}, so $F$ has no step and $\psi(F)=\cT$. \\
Suppose that $j\geq 1$. Recall that the parallelogram polyomino $\cS$, which is the hole of $\cP$, is determined by the north-east paths $S_1$ and $S_2$ with endpoints $(a_0,b_0)$ and $(a_k,b_k)$. Referring to Figure \ref{Figure: proof for equivalance rook}, we consider several cases. 

 \begin{itemize}
     \item[Case 1] Assume that all rooks of $\cT$ are in $\cP_1$. From Remark \ref{Remark: facets vs chains in parallelogram}, there exists a facet $F_1$ of $\Delta(\cP_1)$ of $j$ steps corresponding to $\cT$. Set $F=F_1\cup (S_2\setminus\{(a_0,b_0),(a_k,b_k)\})$. It is easy to see that $F$ is a facet with $j$ steps of $\Delta(\cP)$ and, in particular, $\psi(F)=\cT$, that is the claim. 
     \item[Case 2] Suppose that all rooks of $\cT$ are in $\cP_2$. Applying Remark \ref{Remark: facets vs chains in parallelogram}, there exists a facet $F_2$ of $\Delta(\cP_2)$ of $j$ steps corresponding to $\cT$. Now we examine four sub-cases depending on the placement of the rooks in $\cQ$. 
     \begin{enumerate}
         \item Assume that there is no rook in $\cQ$. Since $j\geq 1$, then at least one rook is in $\cP_2\setminus \cQ$. We denote by $(t_x, t_y)$ and $(r_x, r_y)$ respectively the lower right corner of the cells in $\cP_2$ where the most left and the most right rooks of $\cT$ are placed. We set $L_1=[(1,b_0+1),(1,n)]\cup [(1,n),(a_k-1,n)]$ and we distinguish the following sub-cases. 
         \begin{enumerate}
             \item If $t_y\leq b_0$ and $r_x>a_k$, then $F=\big( F_2\setminus ([(2,t_y),(a_0,t_y)]\cup[(r_x,b_k),(r_x,n-1)])\big)\cup L_1$ (see Figure \ref{Figure: Sur zero, configuration facet});

             \begin{figure}[h!]
		\centering
		\includegraphics[scale=1]{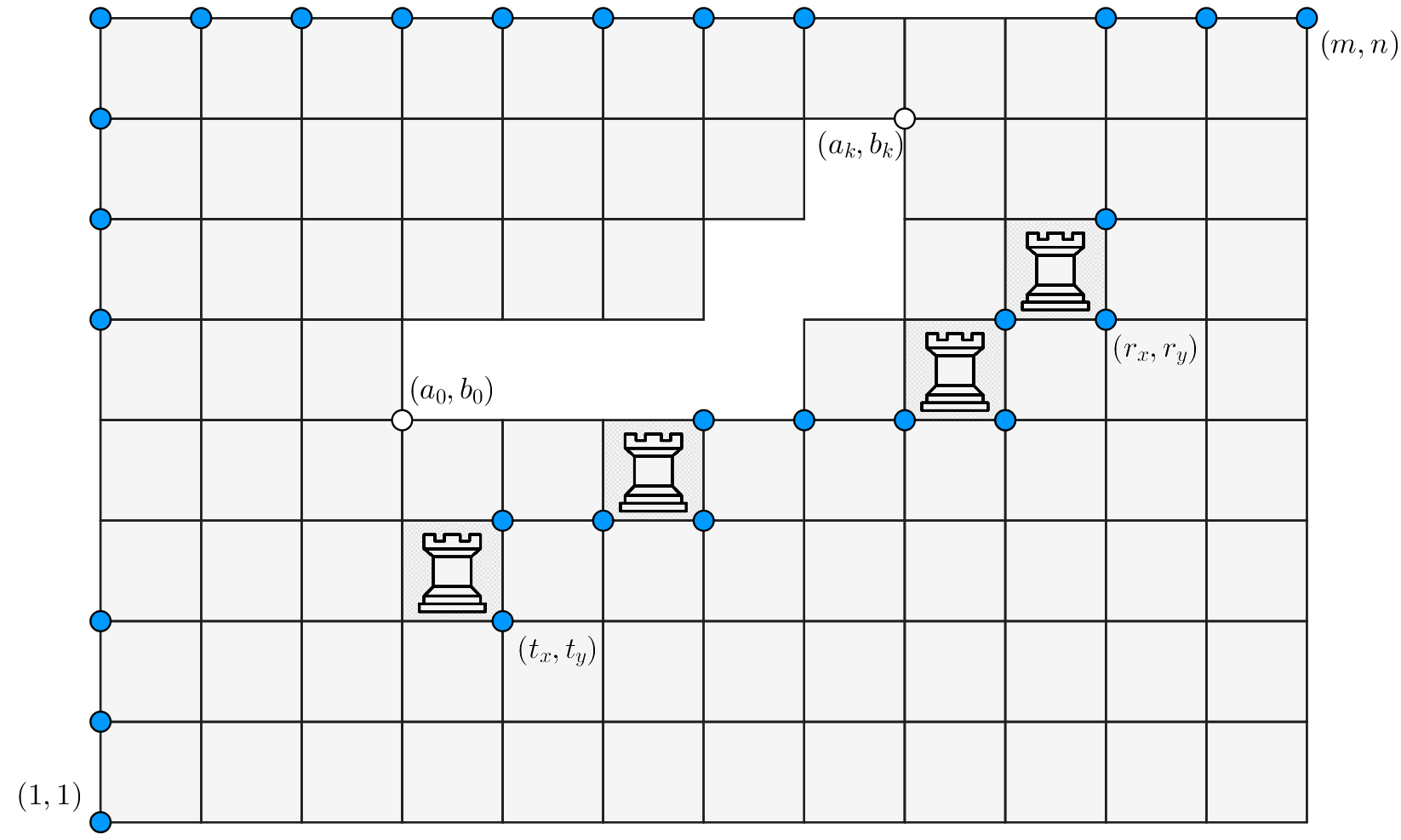}
		\caption{A canonical configuration and the related facet.}
		\label{Figure: Sur zero, configuration facet}
	\end{figure}
 
             \item If $t_y\leq b_0$ and $r_x\leq a_k$, then $F=\big( F_2\setminus ([(2,t_y),(a_0,t_y)]\cup[(a_k,b_k),(a_k,n-1)])\big)\cup L_1$;
             
             \item If $t_y> b_0$ and $r_x\leq a_k$, then $F=\big( F_2\setminus ([(2,b_0),(a_0,b_0)]\cup[(a_k,b_k),(a_k,n-1)])\big)\cup L_1$;
             
             \item If $t_y> b_0$ and $r_x> a_k$, then $F=\big( F_2\setminus ([(2,b_0),(a_0,b_0)]\cup[(r_x,b_k),(r_x,n-1)])\big)\cup L_1$.
         \end{enumerate}
         
         It is easy to see in every previous sub-case that $F$ is a facet of $j$ steps of $\Delta(\cP)$ and $\psi(F)=\cT$, which is the desired claim.
         \item Suppose that at least one rook is in $\cP_{[(1,1),(a_0,b_0)]}$ and no one in $\cP_{[(a_k,b_k),(m,n)]}$. We denote by $(a',b')$ the lower right corner of the cell of $\cP$ where it is placed the most right rook of $\cT$ in $\cP_{[(1,1),(a_0,b_0)]}$. Set $L_2=[(a',b_0+1),(a',n)]\cup [(a',n),(a_k-1,n)]$.
         \begin{itemize}
             \item If no rook is in $\cP_2\setminus \cQ$, then we define $F=\big( F_2\setminus ([(a'+1,b_0),(a_0,b_0)]\cup[(a_k,b_k),(a_k,n-1)])\big)\cup L_2$.  
             
             \item If a rook is in $\cP_2\setminus \cQ$, we indicate by $(t_x, t_y)$ and $(r_x, r_y)$ respectively the lower right corner of the cells in $\cP_2$ where the most left and the most right rooks of $\cT$ are placed. We have the following four sub-cases.
         \begin{enumerate}
             \item If $t_y\leq b_0$ and $r_x>a_k$, then $F=(F_2\setminus [(a'+1,t_y),(a_0,t_y)]\cup[(r_x,b_k),(r_x,n-1)])\cup L_2$ (see Figure \ref{Figure: Sur one, configuration facet}).
             
        \begin{figure}[h!]
		\centering
		\includegraphics[scale=1]{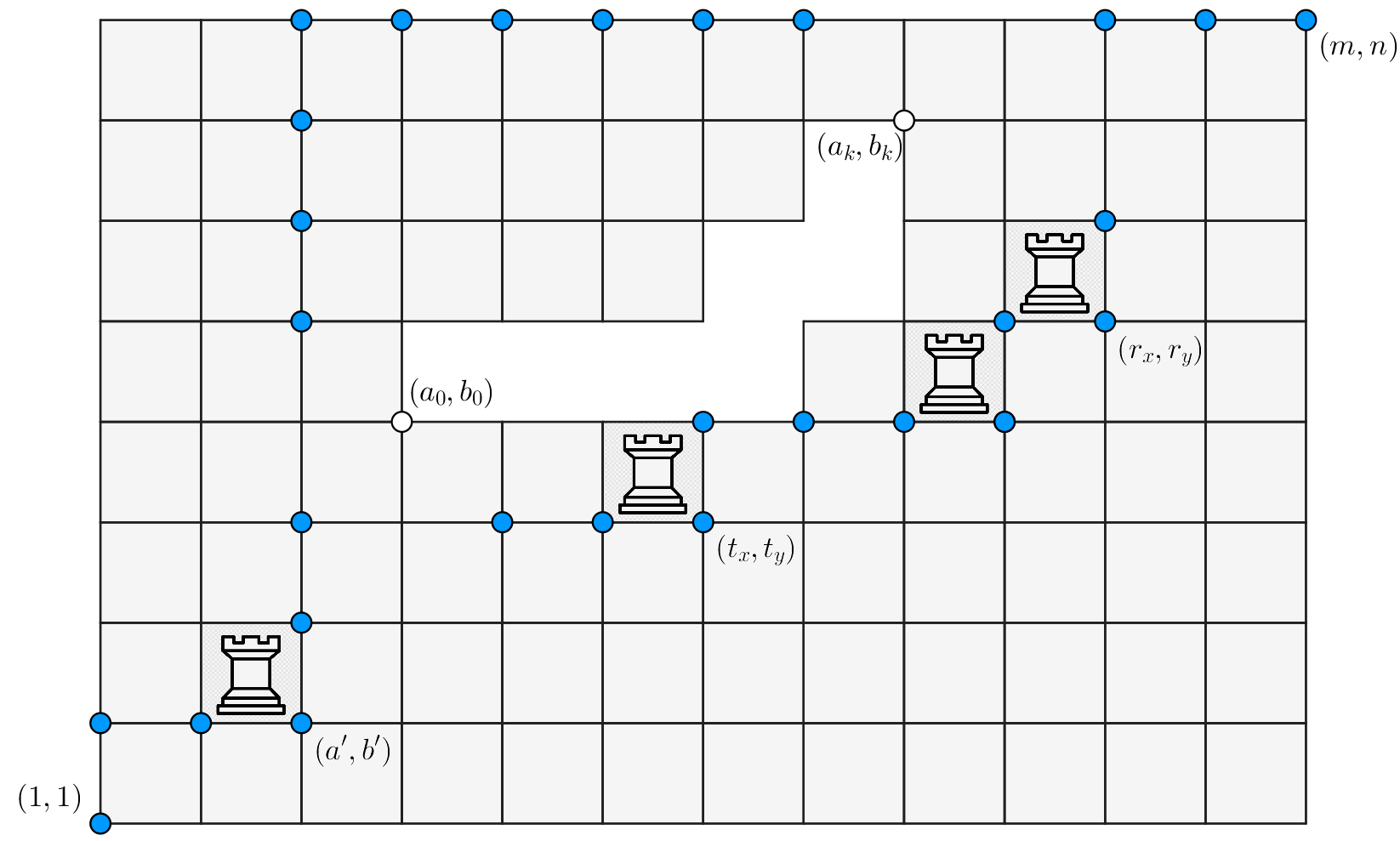}
		\caption{A canonical configuration and the related facet.}
		\label{Figure: Sur one, configuration facet}
	\end{figure}

             \item If $t_y\leq b_0$ and $r_x\leq a_k$, then $F=(F_2\setminus [(a'+1,t_y),(a_0,t_y)]\cup[(a_k,b_k),(a_k,n-1)])\cup L_2$;
             
             \item If $t_y> b_0$ and $r_x\leq a_k$, then $F=(F_2\setminus [(a'+1,b_0),(a_0,b_0)]\cup[(a_k,b_k),(a_k,n-1)])\cup L_2$;
             
             \item If $t_y> b_0$ and $r_x> a_k$, then $F=(F_2\setminus [(a'+1,b_0),(a_0,b_0)]\cup[(r_x,b_k),(r_x,n-1)])\cup L_2$.
         \end{enumerate}

          In every case $F$ is a facet of $j$ steps of $\Delta(\cP)$ and $\psi(F)=\cT$.
         \end{itemize}
         
         \item The case when no rook is placed in $\cP_{[(1,1),(a_0,b_0)]}$ and at least one in $\cP_{[(a_k,b_k),(m,n)]}$ can be proved similarly, as well as when a rook is both in $\cP_{[(1,1),(a_0,b_0)]}$ and in $\cP_{[(a_k,b_k),(m,n)]}$.
         
     \end{enumerate}
     
     \item[Case 3] Now, we can assume that one rook of $\cT$ is in $\cP_1$ and another one in $\cP_2$. Consider the following two cell intervals: $E_{\mathrm{h}}$ is the set of the cells of $\cP$ with lower right corner $(i,b_0)$ for all $i=2,\dots, a_0$ and $E_{\mathrm{v}}$ is that one of the cells of $\cP$ with lower right corner $(a_k,l)$ for all $l=b_k,\dots, n-1$. Here we need to distinguish some cases depending on the placement of the rooks in $\cQ$, $E_{\mathrm{h}}$, or $E_{\mathrm{v}}$. 
     \begin{enumerate}
         \item Suppose firstly that no rook of $\cT$ is in $\cQ\cup E_{\mathrm{h}} \cup E_{\mathrm{v}}$. 
         Consider the parallelogram sub-polyomino $\cK_1=\cP_1\setminus (\cQ\cup E_{\mathrm{h}} \cup E_{\mathrm{v}})$ of $\cP_1$. 
         Note that all rooks of $\cT$ are placed in $\cK_1$ and $\cP_2\setminus \cQ$. From Remark \ref{Remark: facets vs chains in parallelogram} it follows that there exist two facets $K_1$ and $K_2$ respectively of $\Delta(\cK_1)$ and $\Delta(\cP_2)$ corresponding to the canonical configurations in $\cK_1$ and $\cP_2$. Denote by $(t_x,t_y)$ and $(r_x,r_y)$ respectively the lower right corner of the cells in $\cP_2$ where the most left and the most right rooks of $\cT$ are placed. We define $F$ as follows.
         \begin{enumerate}
             \item If $t_y\leq b_0$ and $r_x>a_k$, then $F=K_1\cup (K_2\setminus ([(2,t_y),(a_0,t_y)]\cup [(r_x,b_k),(r_x,n-1)])$ (see Figure \ref{Figure: Sur two, configuration facet});
             
        \begin{figure}[h!]
		\centering
		\includegraphics[scale=1]{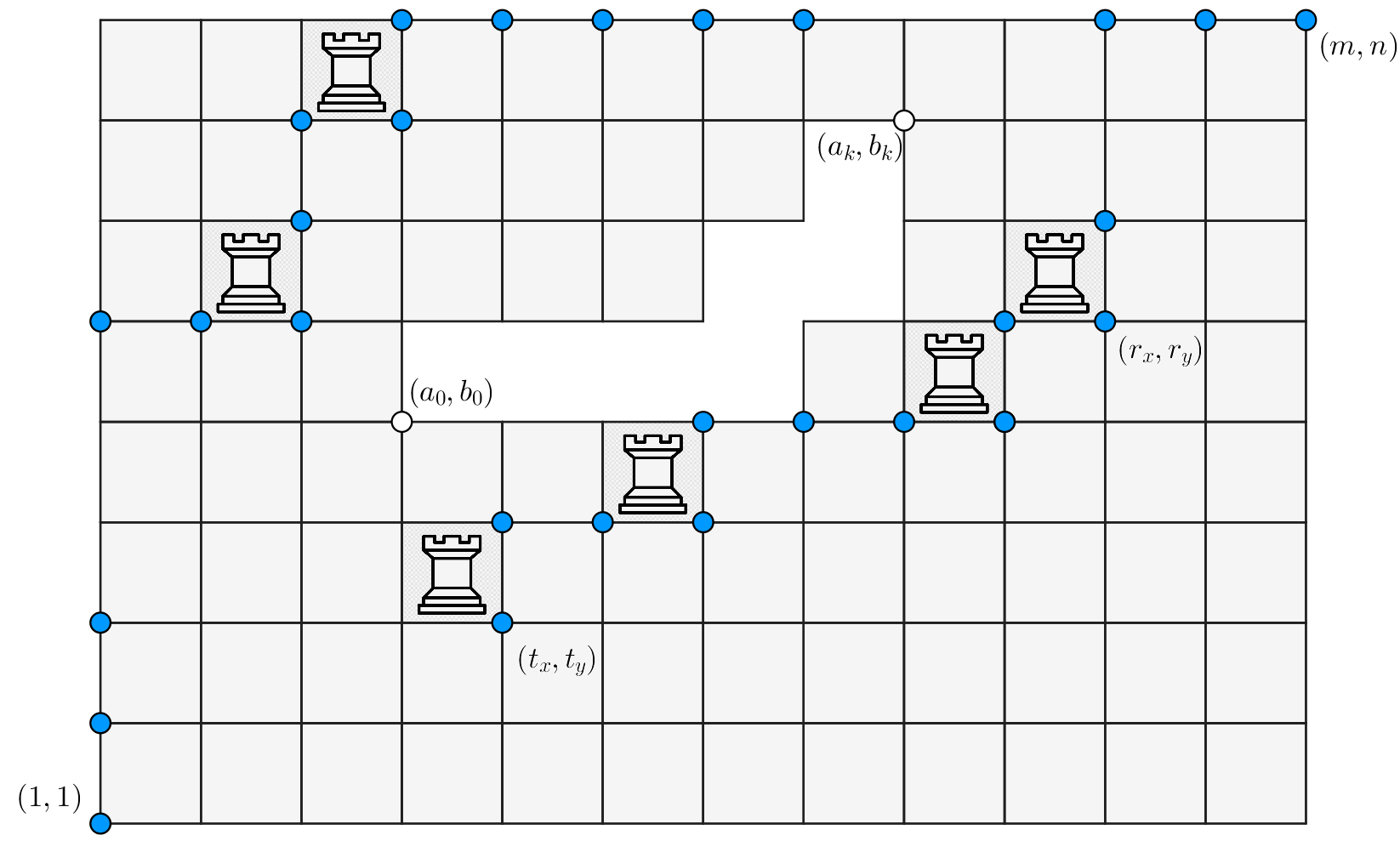}
		\caption{A canonical configuration and the related facet.}
		\label{Figure: Sur two, configuration facet}
	\end{figure}
 
             \item If $t_y\leq b_0$ and $r_x\leq a_k$, then $F=K_1\cup (K_2\setminus ([(2,t_y),(a_0,t_y)]\cup [(a_k,b_k),(a_k,n-1)])$;
             \item If $t_y> b_0$ and $r_x\leq a_k$, then $F=K_1\cup (K_2\setminus ([(2,b_0),(a_0,b_0)]\cup [(a_k,b_k),(a_k,n-1)])$;
             \item If $t_y> b_0$ and $r_x> a_k$, then $F=K_1\cup (K_2\setminus ([(2,b_0),(a_0,b_0)]\cup [(r_x,b_k),(r_x,n-1)])$.
         \end{enumerate}
         It is easy to see that, in every case, $F$ is a facet of $j$ steps of $\Delta(\cP)$ and, moreover, that $\psi(F)=\cT$, which is the claim.
         
         \item Assume now that no rook of $\cT$ is placed in $E_{\mathrm{h}}\cup E_{\mathrm{v}}$ but there exists at least one in $\cQ$. Moreover, we may suppose that a rook is in $\cP_{[(1,1),(a_0,b_0)]}$ and another one in $\cP_{[(a_k,b_k),(m,n)]}$, because the other cases can be proved similarly. Observe that if there are no rooks in $\cP_1\setminus\cQ$ (resp. $\cP_2\setminus\cQ$) then we are in Case 2 (resp. Case 1) so the proof is done. Assume that at least a rook is both in $\cP_1\setminus\cQ$ and in $\cP_2\setminus\cQ$. Let $(t_x,t_y)$ and $(r_x,r_y)$ be respectively the lower right corner of the cells in $\cP_2$ where the most left and the most right rooks of $\cT$ are placed. Denote by $(c_x,c_y)$ the lower right corner of the cell in $\cP_{[(1,1),(a_0,b_0)]}$ where the most right rook of $\cT$ is placed, and by $(d_x,d_y)$ the lower right corner of the cell in $\cP_{[(a_k,b_k),(m,n)]}$ where the most left rook of $\cT$ is placed. Observe that $c_y< t_y$ and $r_x<d_x$. In fact, $c_y\leq t_y$ and $r_x\leq d_x$ from Discussion \ref{Discussion: facet} and, moreover, $c_y\neq t_y$ and $r_x\neq d_x$ since no rook of $\cT$ is placed in $E_{\mathrm{h}}\cup E_{\mathrm{v}}$. Now, we examine the following sub-cases.
         \begin{enumerate}
             \item If $t_y<b_0$ and $r_x>a_k$, then consider $\cP_1$ and the parallelogram polyomino $\cK_2$ given by the north-east paths $[(a_0,t_y),(r_x,t_y)]\cup [(r_x,t_y),(r_x,b_k)]$ and $[(a_0,t_y),(a_0,b_0)]\cup (S_2\setminus\{(a_0,b_0),(a_k,b_k)\}\cup [(a_k,b_k),(r_x,b_k)]$. From Remark \ref{Remark: facets vs chains in parallelogram} there exist two facets $K_1$ and $K_2$ respectively of $\Delta(\cP_1)$ and $\Delta(\cK_2)$ corresponding to the canonical configurations in $\cP_1$ and $\cK_2$. Then we set $F=\big(K_1\setminus([(c_x,t_y+1),(c_x,b_0)]\cup [(a_k,d_y),(r_x-1,d_y)])\big)\cup \big(K_2\setminus\{(a_0,t_y),(r_x,b_k)\}\big)$ (see Figure \ref{Figure: Sur tre, configuration facet}). By construction, it is easy to see that $F$ is a facet of $j$ steps of $\Delta(\cP)$ and, moreover, that $\psi(F)=\cT$.

              \begin{figure}[h!]
		\centering
		\includegraphics[scale=1]{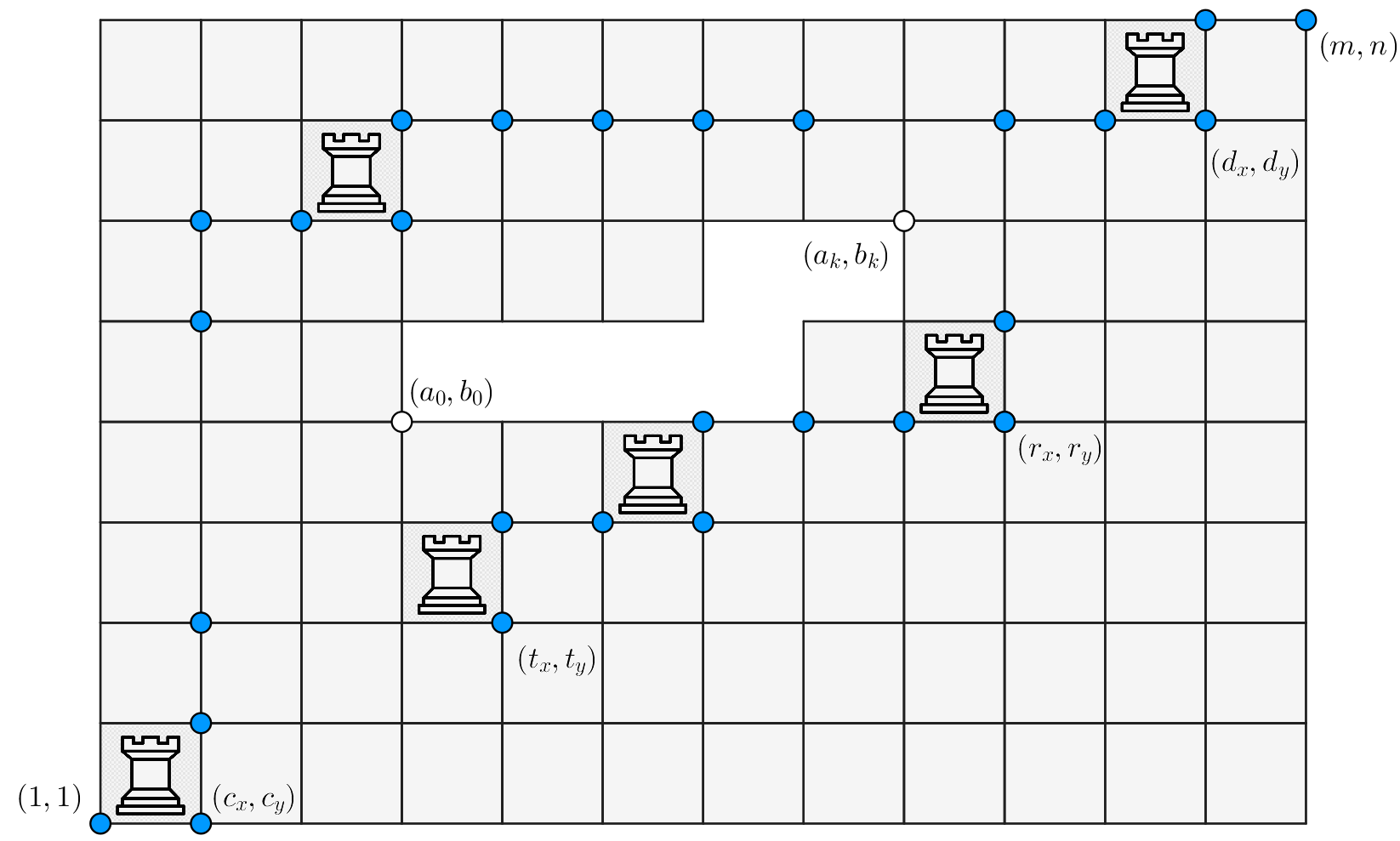}
		\caption{A canonical configuration and the related facet.}
		\label{Figure: Sur tre, configuration facet}
	\end{figure}
 
             \item If $t_y\geq b_0$ and $r_x\leq a_k$, then we need to consider $\cP_1$ and the parallelogram polyomino $\cK_2$ given by the north-east paths $[(a_0+1,b_0-1),(a_k+1,b_0-1)]\cup [(a_k+1,b_0-1),(a_k+1,b_k-1)]$ and $[(a_0+1,b_0-1),(a_0+1,b_0)]\cup (S_2\setminus\{(a_0,b_0),(a_k,b_k)\}\cup [(a_k,b_k-1),(a_k+1,b_k-1)]$. As before, there exist two facets $K_1$ and $K_2$ respectively of $\Delta(\cP_1)$ and $\Delta(\cK_2)$ corresponding to the canonical configurations in $\cP_1$ and $\cK_2$, and in such a case the desired facet is $F=K_1\cup K_2\setminus(\{(a_0+1,b_0-1),(a_k+1,b_k-1)\})$.
             \item If $t_y<b_0$ and $r_x\leq a_k$, then it is sufficient to consider $\cP_1$ and the parallelogram polyomino $\cK_2$ given by the north-east paths $[(a_0,t_y),(a_k+1,t_y)]\cup [(a_k+1,t_y),(a_k+1,b_k-1)]$ and $[(a_0,t_y),(a_0,b_0)]\cup (S_2\setminus\{(a_0,b_0),(a_k,b_k)\}\cup [(a_k,b_k-1),(a_k+1,b_k-1)]$. As before, $K_1$ and $K_2$ are the two facets respectively of $\Delta(\cP_1)$ and $\Delta(\cK_2)$ corresponding to the canonical configurations in $\cP_1$ and $\cK_2$ and  $F=\big( K_1\setminus[(c_x,t_y+1),(c_x,b_0)]\big) \cup \big(K_2\setminus(\{(a_0,t_y),(a_k+1,b_k-1)\})\big)$.
             \item If $t_y\geq b_0$ and $r_x>a_k$, it is clear how the facet $F$ can be defined similarly. 
         \end{enumerate}        
     
         \item Suppose now that there exists a rook in $E_{\mathrm{h}} \cup E_{\mathrm{v}}$. It is not restrictive to assume that there is a rook both in $E_{\mathrm{h}}$ and in $E_{\mathrm{v}}$, one in $\cP_{[(1,1),(a_0,b_0)]}$ and another one in $\cP_{[(a_k,b_k),(m,n)]}$. Moreover, if there is not any rook in $\cP_2\setminus \cQ$ then we are in Case 1, so we can assume that a rook is also in $\cP_2\setminus \cQ$. We denote by $(f_x,b_0)$ and $(a_k,g_y)$ respectively the lower right corner of the rook in $E_{\mathrm{h}}$ and $E_{\mathrm{v}}$, and we set $(t_x,t_y)$, $(r_x,r_y)$, $(c_x,c_y)$ and $(d_x,d_y)$ as before. Here we examine just the case when $t_x<b_0$ and $r_x>a_k$ because the other ones can be proved using the same approach and some considerations as done in (2) of Case 3. Consider now the following four parallelogram polyominoes: $\cB_1$ and $\cB_2$ are the rectangular polyominoes given respectively by $[(1,1),(f_x,t_y)]$ and $[(r_x,d_y),(m,n)]$, $\cG_1$ is the parallelogram sub-polyomino of $\cP_1$ determined by the north-east paths $[(f_x-1,b_0),(f_x-1,g_y+1)]\cup [(f_x-1,g_y+1),(a_k,g_y+1)]$ and $[(f_x-1,b_0),(a_0,b_0)]\cup S_1\cup [(a_k,b_k),(a_k,g_y+1)]$ and $\cG_2$ is the parallelogram sub-polyomino of $\cP_2$ determined by $[(a_0,t_y),(a_0,b_0)]\cup S_2\cup [(a_k,b_k),(r_x,b_k)]$ and $[(a_0,t_y),(r_x,t_y)]\cup [(r_x,t_y),(r_x,b_k)]$. From Remark \ref{Remark: facets vs chains in parallelogram}, there exist four facets $F_1$, $F_2$, $F_3$ and $F_4$ of respectively $\Delta(\cB_1)$, $\Delta(\cB_2)$, $\Delta(\cG_1)$ and $\Delta(\cG_2)$  corresponding to the canonical configurations in that four parallelogram polyominoes. Consider $F=F_1\cup F_2\cup (F_3\setminus \{(f_x-1,b_0),(f_x,b_0),(a_k,g_y),(a_k,g_y+1)\}) \cup (F_4\setminus \{(a_0,t_y),(a_k,b_k)\})$. Look at Figure \ref{Figure: arrangement vs facet} for an example of such a case. We want to point out that when we remove $(f_x-1,b_0),(f_x,b_0)$ from $F_3$, we are removing the step in $F_3$ corresponding to the rook in $E_{\mathrm{h}}$, but it is substituted by $\{(f_x-1,t_y),(f_x,t_y),(f_x,b_0+1)\}$ as step of $F$. The same holds for the vertices $(a_k,g_y),(a_k,g_y+1)$. In light of this, it is easy to see that $F$ is a facet of $\Delta(\cP)$ with $j$ steps and $\psi(F)=\cT$.    
     \end{enumerate}
 \end{itemize}
  Hence $\psi$ is surjective. In conclusion, $\psi$ is bijective.	
	\end{proof}

	   \noindent In Figure \ref{Figure: arrangement vs facet} we figure out a canonical configuration $\cC$ in $\cP$ of seven rooks and the facet of $\Delta(\cP)$ attached to $\cC$.\\
	
	\begin{figure}[h!]
		\centering
		\includegraphics[scale=1]{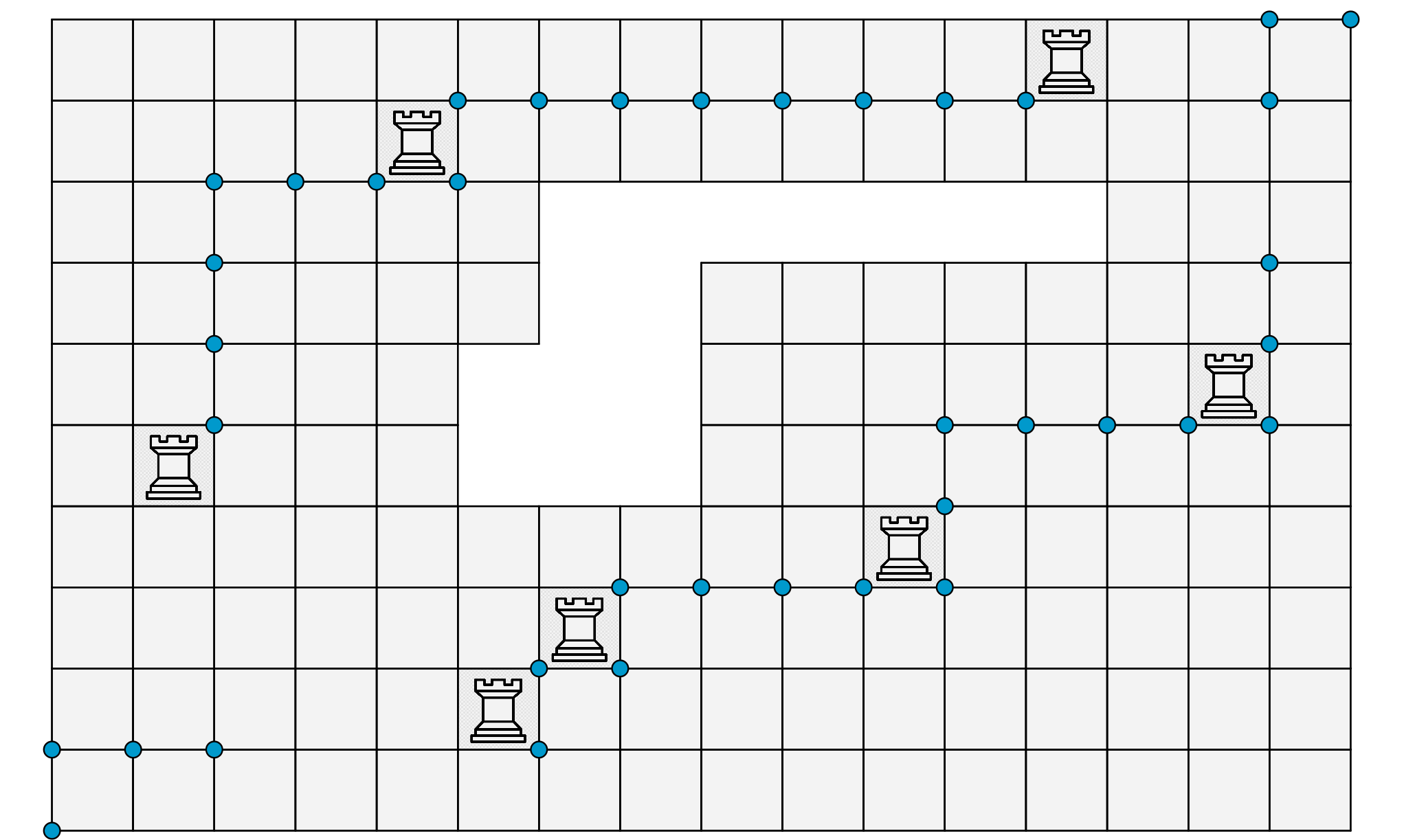}
		\caption{A canonical configuration of seven rooks and the related facet.}
		\label{Figure: arrangement vs facet}
	\end{figure}	 

 \noindent Finally we are ready to prove the main result of this work. 
 
	\begin{thm}\label{mainthm}
	Let $\cP$ be a frame polyomino. The $h$-polynomial of $K[\cP]$ is the switching rook polynomial of $\cP$.
	\end{thm}
	
	\begin{proof}
	From (2) of Theorem \ref{Proposition: The facet less the lower right corner generates the intersection} we know that the $j$-th coefficient of the $h$-polynomial of $K[\cP]$ is the number of the facets of  $\Delta(\cP)$ having $j$ steps. From Theorem \ref{Theorem: bijection between facets and canonical positions}, the latter is equal to the number of the canonical configurations in $\cP$ of $j$ rooks, which is the $j$-th coefficient of the switching rook polynomial of $\cP$. 
	\end{proof}

  \noindent It is established that, if $\cP$ is a frame polyomino, then $K[\cP]$ is a Cohen-Macaulay domain so the Castelnuovo-Mumford regularity of $K[\cP]$ is the degree of the $h$-polynomial of $K[\cP]$, which represents the rook number of $\cP$ for Theorem \ref{mainthm}. 

 \begin{coro}\label{Coro: rook number}
     Let $\cP$ be a frame polyomino. Then the Castelnuovo-Mumford regularity of $K[\cP]$ is the rook number of $\cP$.
 \end{coro}
 
 \noindent We conclude this work observing that \cite[Conjecture 3.2]{Parallelogram Hilbert series} is given just from simple polyominoes. Actually, a frame polyomino is a non-simple polyomino so it is natural to think that \cite[Conjecture 3.2]{Parallelogram Hilbert series} could be extended to every polyomino. 
 
 \begin{conj}\label{Conj}
     Let $\cP$ be a polyomino. The $h$-polynomial of $K[\cP]$ is the switching rook polynomial of $\cP$.
 \end{conj}

%\subsection*{Acknowledgements} The authors would like to warmly thank Professor Ayesha Asloob Qureshi for the deep discussions and for her helpful suggestions and comments, without them this paper would not have been made.  This work was developed while the second author visited her at Sabanci University and he wants to express his gratitude for her hospitality and her nice support.


\begin{thebibliography}{99}
		
		\addcontentsline{toc}{chapter}{\bibname}

        \bibitem{Andrei} C. Andrei, Algebraic Properties of the Coordinate Ring of a Convex Polyomino, The Electronic Journal of Combinatorics, \textbf{28}(1), \#P1.45, 2021.
  
		\bibitem{Bruns_Herzog} W.~Bruns, J.~Herzog, \textit{Cohen-Macaulay rings}, Cambridge University Press, London, Cambridge N.Y., 1993.		

		\bibitem{conca1} A. Conca, \textit{Ladder Determinantal Rings}, Journal of Pure and Applied Algebra, \textbf{98}, 119-134, 1995.
		
		\bibitem{conca2} A. Conca, \textit{Gorenstein ladder determinantal rings}, Journal of the London Mathematical Society, \textbf{54}, 453-474, 1996.
		
		\bibitem{conca3} A. Conca, J. Herzog, \textit{Ladder Determinantal Rings Have Rational Singularities}, Advances in Mathematics, \textbf{132}, 120--147, 1997.
  
		\bibitem{Cisto_Navarra_closed_path} C.~Cisto, F.~Navarra, \textit{Primality of closed path polyominoes}, Journal of Algebra and its Applications, \textbf{22}(2), 2350055, 2023.

        \bibitem{Cisto_Navarra_Veer} C.~Cisto, F.~Navarra, D.~ Veer, \textit{Polyocollection ideals and primary decomposition of polyomino ideals}, arXiv2302.08337, 2023.
  
		\bibitem{Cisto_Navarra_weakly} C.~Cisto, F.~Navarra, R.~Utano, \textit{Primality of weakly connected collections of cells and weakly closed path polyominoes}, Illinois Journal of Mathematics, \textbf{66}, 1-19, 2022.
		
		\bibitem{Cisto_Navarra_CM_closed_path} C.~Cisto, F.~Navarra, R.~ Utano, \textit{On Gr\"obner bases and Cohen-Macaulay property of closed path polyominoes}, The Electronic Journal of Combinatorics,  \textbf{29}, \#P3.54, 2022. 
		
		\bibitem{Cisto_Navarra_Hilbert_series} C.~Cisto, F.~Navarra, R.~ Utano, \textit{Hilbert-Poincar\'e series and Gorenstein property for some non-simple polyominoes}, Bulletin of the Iranian Mathematical Society, \textbf{49}, 2023.
		
		\bibitem{Dinu_Navarra_Konig} R.~Dinu, F.~Navarra, \textit{Non-simple polyominoes of K\"onig type}, arXiv:2210.12665, submitted.

        \bibitem{M2} D. R. Grayson, M. E. Stillman, “Macaulay2: a software system for research in algebraic geometry”, available at \url{http://www.math.uiuc.edu/Macaulay2}.

        \bibitem{Herzog rioluzioni lineari} V.~Ene, J.~Herzog, T.~Hibi, \textit{Linearly related polyominoes}, Journal of Algebraic Combinatorics, \textbf{41}, 949–968, 2015.

    
		\bibitem{L-convessi} V.~Ene, J.~Herzog, A.~A.~Qureshi, F.~Romeo, \textit{Regularity and Gorenstein property of the L-convex polyominoes}, The Electronic Journal of Combinatorics, \textbf{28}(1), \#P1.50, 2021.
  
       \bibitem{golomb} S. W. Golomb, \textit{Polyominoes, puzzles, patterns, problems, and packagings}, Second edition, Princeton University Press, 1994.

        \bibitem{gorla} E.~Gorla, \textit{Mixed ladder determinantal varieties from two-sided ladders}, Journal of Pure and Applied Algebra, \textbf{211}, 433-444, 2007.
        
        \bibitem{Hibi - Herzog Konig type polyomino} J.~Herzog, T.~Hibi, \textit{Finite distributive lattices, polyominoes and ideals of K\"onig type}, arXiv:2202.09643, preprint.
        
        \bibitem{adiajent1} J. Herzog, T. Hibi, \textit{Ideals generated by adjacent 2-minors}, Journal of Commutative Algebra, \textbf{4}, 525-549, 2012.
		
	    \bibitem{2.n} J. Herzog, T. Hibi, F. Hreinsdoottir, T. Kahle, J. Rauh, \textit{Binomial edge ideals and conditional independence statements}, Advances in Applied Mathematics, \textbf{45}, 317-333, 2010.
		
		\bibitem{Simple equivalent balanced} J.~Herzog, S.~S.~Madani, \textit{The coordinate ring of a simple polyomino}, Illinois Journal of Mathematics, \textbf{58}, 981-995, 2014.
		
		\bibitem{def balanced} J.~Herzog, A.~A.~Qureshi, A.~Shikama, \textit{Gr\"obner basis of balanced polyominoes}, Mathematische Nachrichten, \textbf{288}, 775-783, 2015.

        \bibitem{binomial ideals} J. Herzog, T. Hibi, H. Ohsugi, \textit{Binomial Ideals}, Graduate Texts in Mathematics, \textbf{279}, Springer, 2018.
		
		\bibitem{Not simple with localization} T.~Hibi, A.~A.~Qureshi, \textit{Non-simple polyominoes and prime ideals}, Illinois Journal of Mathematics, \textbf{59}, 391-398, 2015.

        \bibitem{adjent 2}S. Hosten, S. Sullivant, \textit{Ideals of adjacent minors}, Journal of Algebra, \textbf{277}, 615-642, 2004.

        \bibitem{Kapl-Riordan}  I. Kaplansky and J. Riordan, \textit{The Problem of Rooks and Its Applications}, Duke Mathematical Journal, \textbf{13}, 259-268, 1946.



        \bibitem{Kummini rook polynomial} M.~Kummini, D.~Veer, \textit{The $h$-polynomial and the rook polynomial of some polyominoes}, The Electronic Journal of Combinatorics,  \textbf{30}(2), \#P2.6, 2023.


        \bibitem{Kummini CD} M.~Kummini, D.~Veer, \textit{The Charney-Davis conjecture for simple thin polyominoes}, Communications in Algebra,  \textbf{51}, 1654-1662, 2022.
		
		\bibitem{Trento} C.~Mascia, G.~Rinaldo, F.~Romeo, \textit{Primality of multiply connected polyominoes}, Illinois Journal of Mathematics, \textbf{64}, 291--304, 2020.

        \bibitem{Trento2} C.~Mascia, G.~Rinaldo, F.~Romeo, \textit{Primality of polyomino ideals by quadratic Gr\"obner basis}, Mathematische Nachrichten, \textbf{295}, 593–606, 2022.

        \bibitem{adiajent3} H. Ohsugi, T. Hibi, \textit{Special simplices and Gorenstein toric rings}, Journal of Combinatorial Theory Series A, \textbf{113}, 718-725, 2006.
	
		\bibitem{Qureshi} A.~A.~Qureshi, \textit{Ideals generated by 2-minors, collections of cells and stack polyominoes}, Journal of Algebra, \textbf{357}, 279-303, 2012.
		
		\bibitem{Parallelogram Hilbert series} A.~A.~Qureshi, G.~Rinaldo, F.~Romeo, \textit{Hilbert series of parallelogram polyominoes}, Research in the Mathematical Sciences, \textbf{9}, 2022.
		
		\bibitem{Simple are prime}  A.~A.~Qureshi, T.~Shibuta, A.~Shikama, \textit{Simple polyominoes are prime}, Journal of Commutative Algebra, \textbf{9}, 413-422, 2017.

  \bibitem{Riordan} J. Riordan, \textit{An introduction to combinatorial analysis}, Wiley Publications in Mathematical
Statistics. John Wiley \& Sons, Inc., New York; Chapman \& Hall, Ltd., London, 1958.
		
		\bibitem{Trento3} G. Rinaldo, and F. Romeo, \textit{Hilbert Series of simple thin polyominoes}, Journal of Algebraic Combinatorics, \textbf{54}, 607-624, 2021.
		
		\bibitem{Shikama} A.~Shikama, \textit{Toric representation of algebras defined by certain non-simple polyominoes}, Journal of Commutative Algebra, \textbf{10}, 265-274, 2018.
				
		\bibitem{Villareal} R.~H.~Villarreal, \textit{Monomial algebras}, Second edition, Monograph and Research notes in Mathematics, CRC press, 2015.
	
	\end{thebibliography}
\end{document}